\documentclass[11pt,reqno]{amsart}
\usepackage{amsthm,amsmath,amssymb,mathrsfs,extarrows,mathtools}
\usepackage{graphicx,cite}
\usepackage{subfig}
\usepackage{epstopdf}
\usepackage{graphics} %% add this and next lines if pictures should be in esp format
\usepackage{epsfig} %For pictures: screened artwork should be set up with an 85 or 100 line screen
\usepackage{tikz}
\usepackage{paralist}
\usepackage{booktabs}
\usepackage{enumerate}
\usepackage{enumitem}
\usepackage{mdwlist}
\usepackage{algorithm,algorithmicx}
\usepackage{color}
\usepackage[colorlinks,linkcolor=blue,anchorcolor=blue,citecolor=red]{hyperref}
\makeatletter
\renewcommand*{\eqref}[1]{%
	\hyperref[{#1}]{\textup{\tagform@{\ref*{#1}}}}%
}
\makeatother
\newtheorem{theorem}{Theorem}[section]

\newtheorem{remark}[theorem]{Remark}

\numberwithin{equation}{section}
\makeatletter

\newcommand{\Rmnum}[1]{\expandafter\@slowromancap\romannumeral #1@}
\makeatother
\newcommand{\me}{\mathrm{e}}
\newcommand{\mi}{\mathrm{i}}
\newcommand{\ms}{\mathbb{S}}
\newcommand{\ps}{\widehat{\mathbb{S}}}
\newcommand{\mm}{\mathbb{M}}
\newcommand{\mr}{\mathbb{R}^2}
\newcommand{\ml}{\mathcal{L}^2}
\newcommand{\fp}{\mathfrak{p}}
\newcommand{\fs}{\mathfrak{s}}

\newcommand{\bi}{\mathbb{I}}
\newcommand{\od}{\overline{D}}
\newcommand{\oo}{\mathcal{O}}

\newcommand{\kp}{\kappa_\mathfrak{p}}
\newcommand{\ks}{\kappa_\mathfrak{s}}

\newcommand{\grad}{\textrm{grad}}
\renewcommand{\div}{\textrm{div}}
\newcommand{\bu}{\boldsymbol{u}}

\newcommand{\bv}{\boldsymbol{v}}

\newcommand{\hx}{\hat{x}}

\newcommand{\uinc}{\boldsymbol{u^i} }
\newcommand{\uincf}{\boldsymbol{u}^{\boldsymbol{i}}_f}

\begin{document}
	
%	\title[Selective focusing of elastic cavities using DORT method]{Far Field Model of Elastic Time Reversal and Application to Selective Focusing on Small Cavities}
		\title[Selective focusing of elastic cavities using DORT method]{Selective focusing of elastic cavities based on the time reversal far field model}

	\author{Jinrui Zhang}
	\address{School of Mathematical Sciences, Zhejiang University,
		Hangzhou, Zhejiang 310027, China}
	\email{zhangjinrui@zju.edu.cn}
	\author{Jun Lai}
	\address{School of Mathematical Sciences, Zhejiang University,
		Hangzhou, Zhejiang 310027, China}
	\email{laijun6@zju.edu.cn}
	 
\thanks{JL was partially supported by the Key Project of Joint Funds for Regional Innovation and Development (No. U21A20425) and the ``Xiaomi Young Scholars" program from Xiaomi Foundation.}
%	
%	\subjclass[2020]{35B40, 35P25, 65R20, 78A46}
%	
%	\keywords{Elastic scattering, multiple scattering, inverse obstacle scattering, time reversal method, fast multipole method}
	
	\begin{abstract}
	This paper is concerned with the inverse time harmonic elastic scattering of multiple small and well-resolved cavities in two dimensions. We extend the so-called DORT method to the inverse elastic scattering so that selective focusing can be achieved on each cavity with far field measurements. A rigorous mathematical justification that relates the corresponding eigenfunctions of the time reversal operator to the locations of cavities is presented based on the asymptotic analysis of the far field operator and decaying property of oscillatory integrals. We show that in the regime $a\ll k^{-1}\ll L$, where $a$ denotes the size of cavity, $k$ is the compressional wavenumber $\kp$ or shear wavenumber $\ks$, and $L$ is the minimal distance between the cavities, each cavity gives rise to five significant eigenvalues and the corresponding eigenfunction generates an incident wave focusing selectively on that cavity. Numerical experiments are given to verify the theoretical result.

	\vspace{1em}
	
\end{abstract}
	
	\maketitle
	
	\section{Introduction}
In the last few years, the inverse elastic scattering problems have attracted extensive attentions due to their applications in nondestructive testing, medical imaging and seismic exploration, etc ~\cite{2015Mathematical,2020An,2018Inverse}.  At the same time, DORT method as a time reversal technique has demonstrated its efficiency in inverse acoustic and electromagnetic problems~\cite{liDecompositionTimeReversal2012a,finkTimeReversalAcousticsBiomedical2003}. This work is concerned with the time harmonic inverse elastic scattering by multiple cavities embedded in a homogeneous and isotropic elastic medium in two dimensions and extends the DORT method to the inverse elastic scattering with a rigorous mathematical justification.

Over the years, many methods have been proposed to solve the inverse elastic obstacle problems, including both the iterative type and direct imaging methods. For instance, in \cite{le2015domain}, a gradient descent method based on the domain derivative of elastic scattering was developed for reconstructing unknown elastic obstacles. In \cite{BAO2018263}, the Fréchet derivative of the near-field operator with respect to the boundary of an anisotropic scatterer was derived and used in a descent algorithm to recover the interface. As a direct imaging method, linear sampling based on the factorization of far field operator has been proposed for the inverse elastic obstacle scattering in \cite{alvesFarfieldOperatorElastic2002}. In ~\cite{huInverseProblemsArising2012a}, factorization methods have been extensively studied for the reconstruction of finitely many point-like and extended elastic rigid obstacles. More approaches on the inverse elastic scattering problems can be found in  ~\cite{arensLinearSamplingMethods2001,huRecoveringComplexElastic2014,huDirectInverseTimeharmonic2020}, etc. 

Despite all these efforts, there are still many challenges for the inverse elastic scattering of multiple obstacles, including the existence of many local minimums via optimization based methods and the difficulties to distinguish different scatterers during the reconstruction. In the present paper, we are concerned with the so-called DORT method (french acronym for Decomposition of the Time Reversal Operator), which is an experimental technique used to focus waves selectively on small and well resolved scatterers (i.e. when multiple scattering can be neglected). This method was first developed by Prada and Fink \cite{pradaEigenmodesTimeReversal1994} in the context of ultrasonics (see \cite{pradaTimeReversalTechniques2002} for an overview) and has been widely used in acoustics and electromagnetics. In the frequency domain, the process of DORT method is simply using a time reversal mirror (TRM), composed of an array of transducers, to emit an incident wave in a homogeneous and nondissipative medium containing some unknown obstacles. The measured far field is then conjugated, which is equivalent to reversing time in the time harmonic case, and re-emitted. The time reversal operator $T$ is obtained by iterating this procedure twice. It turns out the eigenelements of the time reversal operator carry important information on the scatterers contained in the propagation medium. More precisely, in the DORT method, the number of nonzero significant eigenvalues of $T$ is directly related to the number of scatterers, and the corresponding eigenvector generates an incident wave that selectively focuses on each scatterer.

From the mathematical point of view, a rigorous justification of the above result has been given in Hazard and Ramdani \cite{hazardSelectiveAcousticFocusing2004} for the 3D acoustic scattering problem by small sound-soft scatterers using a far field model. In this model, the TRM was supposed to be continuous and located at infinity. In \cite{benamarNumericalSimulationAcoustic2007}, the interactions between time reversal mirror and the obstacles are taken into account. Then, it has been extended to the case of a 2D closed acoustic waveguide in \cite{pinconSelectiveFocusingSmall2007}, the case of perfectly conducting electromagnetic scatterers in \cite{antoineFarFieldModeling2008} and the case of rigid elastic particles in \cite{laiFastInverseElastic2022}. In this work, we are interested in the case of 2D elastic scattering by small cavities.  Using a far field model for time reversal, we prove that in the regime $a\ll k^{-1} \ll L$, where $a$ denotes here the size of cavity, $k$  is the compressional wavenumber $\kp$ or shear wavenumber $\ks$, $L$ is the minimal distance between the cavities, each cavity gives rise to five significant eigenvalues. The result is slightly surprising, as one might expect it is six as a direct consequence from acoustics~\cite{burkardFarFieldModel2013}. This is due to the interconnection between the compressional and shear wave of elastic scattering, where the eigensystem analysis is much more involved compared to that of acoustics and electromagnetics. Furthermore, we also show each corresponding eigenfunction generates an incident wave focusing selectively on that cavity. We start from the analysis of a small elastic disk, where the five significant eigenvalues can be explicitly calculated based on the Mie theory. Then extend the result to general shaped small cavities by combining the elastic potential theory and the asymptotic analysis for weak elastic scattering. 

The paper is organized as follows. In Section 2, we formulate the scattering and inverse  scattering problems of multiple elastic cavities in two dimensions. In Section 3, we show the global focusing properties using the eigenfunctions of the time reversal operator in the case of extended obstacles. Section 4 starts with the eigensystem of a single disk and then gives the asymptotic analysis for the far field operator of small and distant cavities. Section 5 is devoted to the mathematical justification of the DORT method in the case of general shaped small elastic cavities. Numerical simulations are presented in Section 6 to verify the theoretical results. Finally, the paper is concluded in Section 7.

\section{Problem formulation}

Consider $M$ well-separated elastic cavities (or particles) in two dimensions, denoted by $D_1,D_2,\dots, D_M$, as shown in figure \ref{geometry}. Assume their boundaries $\Gamma_{1}, \Gamma_{2}, \dots, \Gamma_{M}$ are at least $C^2$ smooth. Let $D=D_{1}\cup D_{2}\cup\cdots\cup D_{M}$ and $\Gamma=\Gamma_{1}\cup\Gamma_{2}\cup\cdots\cup\Gamma_{M}$.  Denote $\nu$ the unit exterior normal vector on $\Gamma$. The exterior domain $\mr\setminus \od$ is assumed to be filled with a homogeneous and isotropic elastic medium with a unit mass density (i.e. $\rho = 1$). Let the particles be illuminated by a time harmonic incident wave $\uinc$.  The displacement of the total field $\boldsymbol u$ consists of the incident field $\uinc$ and the scattered field $\boldsymbol v$, i.e., $ \boldsymbol u=\boldsymbol{u^i}+\boldsymbol v$, and satisfies the Navier equation
\begin{equation}\label{navierequ}
	\mu\Delta\boldsymbol{u}+(\lambda+\mu)\nabla\nabla\cdot\boldsymbol{u}
	+\omega^2\boldsymbol{u}=0,\mbox{ in } \mathbb{R}^2\setminus \overline{D}, 	
\end{equation}
where $\omega>0$ is the angular frequency and $\lambda, \mu$ are the
Lam\'{e} constants satisfying $\mu>0, \lambda+\mu>0$.
For a nonzero vector $x\in\mr$, let us introduce two unit vectors: $\hx:=x/|x|$, and
$\hx^\bot$, which is obtained by rotating $\hx$ anticlockwise by $\pi/2$.
\begin{figure}
	\centering
	\includegraphics[scale=0.25]{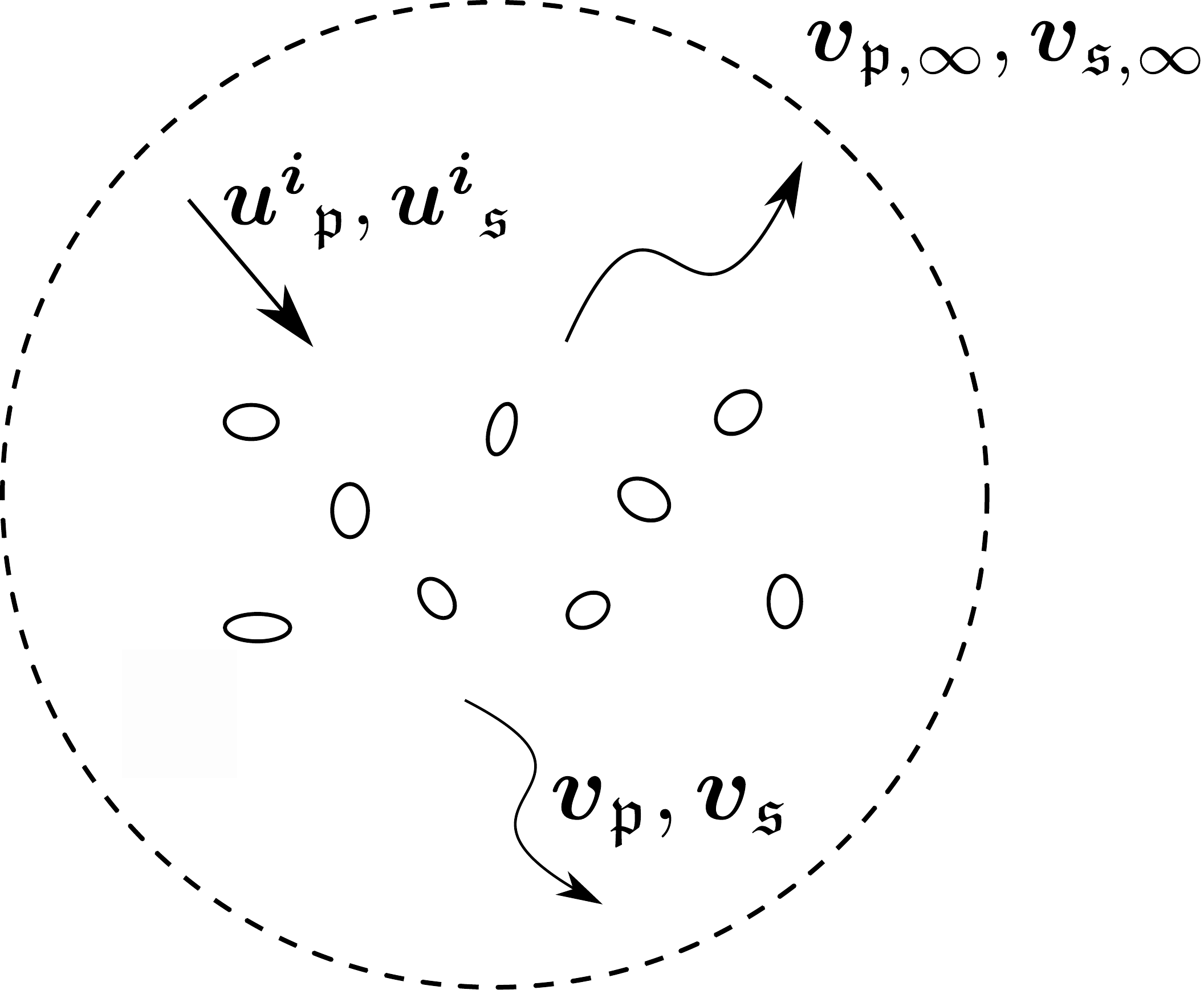}
	\caption{Inverse scattering of multiple elastic cavities using the far field measurement.}\label{geometry}
\end{figure}

In addition to the usual differential operators \textit{grad} and \textit{div}, we also need
\[\grad^\bot \bu=\nabla^\bot \bu:=\left[-\frac{\partial u}{\partial x_2},\frac{\partial u}{\partial x_1}\right]^\top,\qquad \div^\bot\bu=\nabla^\bot\cdot\bu:=\frac{\partial u_2}{\partial x_1}-\frac{\partial u_1}{\partial x_2}.\]
Define the linearized strain tensor by 
\[\varepsilon(\bu):=\frac{1}{2}(\nabla \bu+\nabla \bu^\top)\in\mathbb{C}^{2\times2},\]
where $\nabla \bu$ and $\nabla \bu^\top$ stand for the Jacobian matrix of $\bu(x)\in \mathbb{C}^2$ and its adjoint, respectively. By Hooke’s law, the strain tensor is related to the stress tensor via the identity
\begin{align}\label{sigmau}
	\sigma(u)=\lambda(\nabla\cdot \bu)\mathbb{I}+2\mu\varepsilon(\bu)\in\mathbb{C}^{2\times2},
\end{align}
where $\mathbb{I}$ is the $2\times 2$ identity matrix. The surface traction on $\Gamma$ is defined as
\begin{eqnarray}
	%	T_\nu \bu: = \nu\cdot\sigma(\bu)=2\mu \nu\cdot \grad\bu + \lambda \nu\div\bu - \mu\nu^\bot\div^\bot\bu.
	T_\nu \bu: = \nu\cdot\sigma(\bu)=2\mu \nu\cdot \nabla\bu + \lambda \nu\nabla\cdot\bu - \mu\nu^\bot\nabla^\bot\cdot\bu.
\end{eqnarray}
Since elastic cavities are traction-free, it holds 
\begin{eqnarray}\label{rigidboundary}
	T_\nu\boldsymbol{u}=0,\quad {\rm on}~\Gamma.
\end{eqnarray}

The incident wave $\boldsymbol{u^i}$ is given as a linear combination of a longitudinal plane wave $$\boldsymbol{u^i}(x,\alpha,f_{\fp}(\alpha),0):=\boldsymbol{u^i_{\fp}}(x)=f_{\fp}(\alpha) \alpha\mathrm{e}^{\mathrm{i} \kappa_{\mathfrak p}\alpha\cdot x},$$ and a transversal plane wave $$\boldsymbol{u^i}(x,\alpha,0,f_{\fs}(\alpha)):=\boldsymbol{u^i_{\fs}}(x)=f_{\fs}(\alpha)\alpha^\bot
\mathrm{e}^{\mathrm{i}\kappa_{\mathfrak s} \alpha\cdot x},$$
where $\alpha=(\cos\theta, \sin\theta)^\top$ is the unit propagation vector, and 
$f_{\fp}(\alpha),f_{\fs}(\alpha)\in \mathbb{C}$ are the longitudinal and transversal strength. 
Note that $\boldsymbol{u^i_{\fp}}$ and $\boldsymbol{u^i_{\fs}}$ are also called the compressional and shear incident wave, respectively, with compressional wavenumber $\kp$ and shear wavenumber $\ks$ defined by
\[
\kappa_{\mathfrak p}=\frac{\omega}{\sqrt{\lambda+2\mu}},\quad
\kappa_{\mathfrak s}=\frac{\omega}{\sqrt{\mu}}. 
\]
Throughout the paper, we assume $\kappa_{\mathfrak p}$ and $\kappa_{\mathfrak s}$ are both positive real and on the same order, i.e. $\kappa_{\mathfrak p}\sim \kappa_{\mathfrak s}$.

It is easy to verify that the scattered field $\boldsymbol v$ satisfies
the boundary value problem
\begin{equation}\label{scatteredfield}
	\begin{cases}
		\mu\Delta\boldsymbol{v}+(\lambda+\mu)\nabla\nabla\cdot\boldsymbol{v}
		+\omega^2\boldsymbol{v}=0\quad &{\rm in}~
		\mr\setminus\od,\\
		T_\nu\boldsymbol{v}=-T_\nu\boldsymbol{u^i}\quad &{\rm on}~\Gamma.
	\end{cases}
\end{equation}
By Helmholtz decomposition, any solution $\bv$ to \eqref{scatteredfield} can be
decomposed into the form
\begin{align}\label{helmdec}
	% 	\bv = \bv_\fp+\bv_\fs,\quad\boldsymbol v_{\mathfrak p}=-\frac{1}{\kappa_{\mathfrak p}^2}\grad~\div\boldsymbol v,\quad
	% 	\boldsymbol v_{\mathfrak s}=-\frac{1}{\kappa_{\mathfrak s}^2}\grad^\bot\div^\bot\boldsymbol v, 
	\bv = \bv_\fp+\bv_\fs,\quad\boldsymbol v_{\mathfrak p}=-\frac{1}{\kappa_{\mathfrak p}^2}\nabla\nabla\cdot\boldsymbol v,\quad
	\boldsymbol v_{\mathfrak s}=-\frac{1}{\kappa_{\mathfrak s}^2}\nabla^\bot\nabla^\bot\cdot\boldsymbol v, 
\end{align}
where $\bv_\fp,\bv_\fs$ are known as the  compressional and shear wave components of $\boldsymbol v$, respectively. 
In addition,  the scattered field is  required to satisfy
the Kupradze--Sommerfeld radiation condition
\[
\lim_{r\to\infty}\sqrt{r}(\partial_r\boldsymbol v_{\mathfrak p}-\mathrm{i}\kappa_{\mathfrak p}\boldsymbol v_{\mathfrak p})=0,\quad
\lim_{r\to\infty}\sqrt{r}(\partial_r\boldsymbol v_{\mathfrak s}-\mathrm{i}\kappa_{\mathfrak s}\boldsymbol v_{\mathfrak s})=0,\quad r=|x|.
\]

The fundamental solution to the Navier equation \eqref{navierequ} in the two dimensional free space is given by
\begin{align}\label{greenfun}
	\Phi(x,y) = \frac{\mi\ks^2}{4\omega^2}H_0^{(1)}(\ks|x-y|)\bi+\frac{i}{4 \omega^2}\nabla\nabla^\top \left[H_0^{(1)}(\ks|x-y|)-H_0^{(1)}(\kp|x-y|)\right].
\end{align}

%The traction operator $T_\nu$ on $\Gamma$ is defined by 
%	\begin{eqnarray}
	%		T_\nu: = 2\mu \nu\cdot \nabla + \lambda \nu\nabla\cdot + \mu\nu\times {\bf curl}.
	%	\end{eqnarray}
Based on the Betti's formula \cite{alvesFarfieldOperatorElastic2002} and the boundary condition \eqref{rigidboundary}, we can represent the scattered field $\boldsymbol v$ through the double layer integral formulation, 
\begin{eqnarray}\label{singleintrep}
	\bv(x) =\int_{\Gamma} [T_{\nu(y)}\Phi(x,y)]^\top\bu(y) ds_y, \quad x\in 	\mathbb{R}^2\setminus\overline{D}.
\end{eqnarray}
As $|x|\rightarrow\infty$, the asymptotic behavior of the elastic scattered field $\bv$ is given by 
\begin{equation}\label{farfield}
	\bv(x) = \frac{e^{\mi \kp |x|}}{\sqrt{|x|}} \bv_{\fp,\infty}(\hx)\hx+ \frac{e^{\mi \ks |x|}}{\sqrt{|x|}} \bv_{\fs,\infty}(\hx)\hx^\bot +\mathcal{O}\left(|x|^{-3/2}\right),
\end{equation}
where   $\bv_{\fp,\infty}$ and $\bv_{\fs,\infty}$ are defined on the unit circle  $\ms$ with $\hx=x/|x|$ and known as the compressional and shear wave far field pattern, respectively. 
%	 We will usually consider the pair $(\bv_{\fp,\infty},\bv_{\fs,\infty})$ as an element of $\ml$, however.

Based on the asymptotic behavior of the fundamental solution \eqref{greenfun}, it can be verified that
\begin{align}
	%		\begin{split}
		\begin{aligned}
			\bv_{\fp,\infty}(\hx)\hx &=\frac{\me^{\mi\pi/4}}{\sqrt{8\pi\kp}} \frac{\kp^2}{\omega^2}\int_{\Gamma}[T_{\nu(y)}(\hx\otimes\hx)\me^{-\mi\kp \hx\cdot y}]^\top  \bu(y) ds_y,\\
			\bv_{\fs,\infty}(\hx)\hx^\bot &= \frac{\me^{\mi\pi/4}}{\sqrt{8\pi\ks}} \frac{\ks^2}{\omega^2}\int_{\Gamma}[T_{\nu(y)}(\bi-\hx\otimes\hx)\me^{-\mi\ks \hx\cdot y}]^\top  \bu(y) ds_y,
		\end{aligned}
		%		\end{split} 
\end{align}
where `$\otimes$' denotes the outer product of two vectors.
%	The forward  problem for the elastic scattering of multiple particles is: 
%	\begin{itemize}
	%		\item Given the incident wave $\uinc$ and the geometric information of $D_1, \cdots, D_M$, find the far field pattern $\bv_{\fp,\infty}(\hx)$ and $\bv_{\fs,\infty}(\hx)$ of the scattered field $\bv$. 
	%	\end{itemize}
The inverse problem of elastic scattering of multiple cavities is:
\begin{itemize}
	\item Based on the far field pattern $\bv_{\fp,\infty}(\hx)$ and $\bv_{\fs,\infty}(\hx)$ from different incident directions,  recover the geometric information of $D_1, D_2, \cdots, D_M$, including locations and shapes (see figure \ref{geometry}). 
\end{itemize}

In the following sections, we will first introduce the DORT method to the inverse elastic scattering problem and show the global focusing property in the general case. Then give a mathematical justification for the selective focusing of small and distant cavities.

\section{Global focusing based on TRM}

Let us introduce the vector space $\mathcal{L}^2:=[L^2(\ms)]^2$, which is defined on the unit circle $\ms$ and equipped with the inner product
\begin{eqnarray}\label{innerprod}
	(f,g) = \frac{\omega}{\kp}\int_{\ms} f_{\fp}(\alpha) \overline{g_{\fp}(\alpha)}ds_\alpha + \frac{\omega}{\ks}\int_{\ms}f_{\fs}(\alpha) \overline{g_{\fs}(\alpha)}ds_\alpha,
\end{eqnarray}
where $f=(f_{\fp}, f_{\fs})$ and $g = (g_{\fp},g_{\fs})$.	An element $f=(f_\fp,f_\fs)\in\ml$ will most often denote the compressional and shear far-field patterns, respectively. Denote the pair of far-field patterns $$\big(\bv_{\fp,\infty}(\hx,\alpha,f_{\fp}(\alpha),f_{\fs}(\alpha)), \bv_{\fs,\infty}(\hx,\alpha,f_{\fp}(\alpha),f_{\fs}(\alpha))\big)$$ of the corresponding scattered field by $\bv_\infty(\hx,\alpha,f_\fp(\alpha),f_\fs(\alpha))$, which is radiated by the incident wave  $$\boldsymbol{u^i}(x,\alpha,f_{\fp}(\alpha),f_{\fs}(\alpha)) = \boldsymbol{u^i}(x,\alpha,f_{\fp}(\alpha),0)+\boldsymbol{u^i}(x,\alpha,0,f_{\fs}(\alpha)).$$
%	Define the $L^2$ space 
%	\begin{eqnarray}
	%		L^2_{\fp} = \{f_{\fp}: \ms \rightarrow \mathbb{C}^2\ |\  f_{\fp}(\alpha)\times \alpha = 0 , |f_{\fp}|\in L^2(\ms)\}
	%	\end{eqnarray}
%	of the longitudinal vector fields on $\ms$ and the $L^2$ space 
%	\begin{eqnarray}
	%		L^2_{\fs} = \{f_{\fs}: \ms \rightarrow \mathbb{C}^2\ |\ f_{\fs}(\alpha)\cdot \alpha =0, |f_{\fs}|\in L^2(\ms)\}
	%	\end{eqnarray}
%	of the transversal vector fields, where $|\cdot|$ is the Euclidean norm in $\mathbb{C}^3$. 

The elastic Herglotz wave with kernel $f\in \ml $ has the form
\begin{eqnarray}\label{herglotz}
	\uincf(x) =\me^{-\mi\pi/4}\int_{\ms} \left(\sqrt{\frac{\kp}{\omega}}\me^{\mi \kp \alpha\cdot x} f_{\fp}(\alpha)\alpha +\sqrt{\frac{\ks}{\omega}} \me^{\mi \ks\alpha \cdot x} f_{\fs}(\alpha)\alpha^\bot\right) ds_{\alpha}, 
	%		\uincf(x) =\int_{\ms} \alpha\me^{\mi \kp \alpha\cdot x} f_{\fp}(\alpha) +\alpha^\bot\me^{\mi \ks\alpha \cdot x} f_{\fs}(\alpha)ds_{\alpha}, 
\end{eqnarray}
which is a superposition of plane waves and satisfies the Navier equation entirely. We remark here the elastic Herglotz wave is multiplied by some unusual coefficients to match the far field operator introduced in the following.  By linearity, the corresponding far field operator $$F:\ml \rightarrow \ml$$  due to the incident wave $\uincf(x)$ is defined by
\begin{eqnarray}
	F(f)(\hx) :=\me^{-\mi\pi/4}\int_{\ms}\bv_\infty(\hx,\alpha,\sqrt{\kp/\omega}f_\fp(\alpha),\sqrt{\ks/\omega}f_\fs(\alpha)) ds_{\alpha}, 
	%		F(f)(\hx) :=\int_{\ms}\bv_\infty(\hx,\alpha,f_\fp(\alpha),f_\fs(\alpha)) ds_{\alpha}, 
\end{eqnarray}
%	where $\bv_{\infty}(\hx,\alpha,f_{\fp}(\alpha),f_{\fs}(\alpha)) =\left(\bv_{\fp,\infty}(\hx,\alpha,f_{\fp}(\alpha),f_{\fs}(\alpha)), \bv_{\fs,\infty}(\hx,\alpha,f_{\fp}(\alpha),f_{\fs}(\alpha))\right)$.
for $f=(f_\fp,f_\fs)$. By superposition, $F(f_\fp,f_\fs)$ is the far-field pattern by the scattering of the elastic Herglotz incident wave with kernel $f\in\ml$.
It is easy to see that the far field operator $F$ is compact since the kernel is smooth. Using the reciprocity relation of elastic wave, one can show the following result~\cite{arensLinearSamplingMethods2001}.
\begin{theorem}
	The far field operator $F:\ml \rightarrow \ml$ is a compact and normal operator. Its adjoint operator $F^*:\ml\rightarrow\ml$ with respect to the inner product \eqref{innerprod} is given by
	$$F^*f = \overline{RFR\overline{f}}, \quad \forall f \in \ml,$$
	where $R$ is the symmetry operator defined by $Rf(\alpha)=f(-\alpha), \alpha\in \ms$.
\end{theorem}

It is worth mentioning that similar result also holds for the acoustic and electromagnetic scattering~\cite{coltonInverseAcousticElectromagnetic2019}. We are now able to define the time reversal operator $T$.  First let us measure the far field of the scattered field due to the Herglotz wave $\uincf$ with $f\in \ml$, and then use the conjugate of the far field as the kernel $g$ of a new Herglotz wave. In other words,
\begin{eqnarray*}
	g = \overline{RFf}.
\end{eqnarray*}   
The symmetry operator $R$ is used here in order to reemit the wave from the opposite of the measured direction. The time reversal operator $T$ is then obtained by iterating this cycle twice
\begin{eqnarray}
	T f = \overline{RF g} = \overline{RF \overline{RFf}}.
\end{eqnarray} 
It holds the following property for the time reversal operator $T$.
\begin{theorem}\label{normal}
	The time reversal operator $T$ is compact, self-adjoint and positive. It is defined as an operator from $\ml$ to itself with
	\begin{eqnarray}
		Tf = FF^*f = F^*F f. 
	\end{eqnarray}
	The nonzero eigenvalues of $T$ are exactly positive numbers
	$|\lambda_1|^2\ge |\lambda_2|^2 \ge \cdots >0$
	where the sequence $(\lambda_j)_{j\ge 1}$ denotes the nonzero complex eigenvalue of the far field operator $F$. The corresponding eigenfunctions $(f_j)_{j\ge 1}$ of $F$ are exactly the eigenfunctions of $T$. If non-trivial solutions to the Navier equation in $D$ with traction free boundary condition do not exist, then $(f_j)_{j\ge 1}$ form a complete orthonormal system in $L^2_{\fp}(\ms) \times L^2_{\fs}(\ms) $. 
\end{theorem}

Readers are referred to \cite{arensLinearSamplingMethods2001} for a detailed proof. The time reversal method is to illuminate an obstacle with Herglotz waves with kernel $f$ corresponding to an eigenvector of $F$ (or $T$) with non-zero eigenvalue. In particular, the Herglotz wave generated by $f$ with $\lambda\ne 0$ will automatically focus on the obstacles, as shown by the following theorem.
\begin{theorem}\label{globalf}
	Let $\lambda\ne 0$ be an eigenvalue of $F$ and $f\in \ml$ be an eigenvector of $F$ associated with $\lambda$. Denote $\bu_{f}$ the total elastic field due to the Herglotz incident wave $\uincf$. 
	%		 with kernel $f=(f_\fp(\alpha),f_\fs(\alpha))\in\ml$
	Then, the Herglotz incident wave $\uincf$ with kernel $f\in\ml$ has the following form
	\begin{eqnarray*}
		\uincf =&  \frac{1}{\lambda\sqrt{8\pi\kp}}\left(\frac{\kp}{\omega}\right)^\frac{5}{2} \int_{\Gamma_D} \big[T_\nu(y)E(\kp,x,y)\big]^\top  \bu_{f} ds_y+ \frac{1}{\lambda\sqrt{8\pi\ks}}\left(\frac{\ks}{\omega}\right)^\frac{5}{2} \int_{\Gamma_D}\big[T_\nu(y)H(\ks,x,y)\big]^\top\bu_{f} ds_y,
	\end{eqnarray*}
	where $E(\kp,x,y)$ and $H(\ks,x,y)$ are given in \eqref{idenE} and \eqref{idenH}.
\end{theorem}
\begin{proof}
	Since $f=(f_{\fp},f_{\fs})\in \ml$ is an eigenvector of $F$ with eigenvalue $\lambda\ne0$, it holds
	\begin{align*}
		f_{\fp} (\hx)\hx = & \frac{\me^{-\mi\pi/4}}{\lambda}\int_{\ms}\bv_{\fp,\infty}(\hx,\alpha,\sqrt{\kp/\omega}f_\fp(\alpha),\sqrt{\ks/\omega}f_\fs(\alpha)) ds_{\alpha} \notag \\
		= & \frac{1}{\lambda\sqrt{8\pi\kp}}\frac{\kp^2}{\omega^2}
		\int_{\mathbb{S}} \int_{\Gamma_D}[T_{\nu(y)}(\hx\otimes\hx)\me^{-\mi\kp \hx\cdot y}]^\top \bu(y,\alpha,\sqrt{\kp/\omega}f_\fp(\alpha),\sqrt{\ks/\omega}f_\fs(\alpha)) ds_yds_\alpha \notag \\
		=&  \frac{1}{\lambda\sqrt{8\pi\kp}}\frac{\kp^2}{\omega^2}\int_{\Gamma_D}[T_{\nu(y)}(\hx\otimes\hx)\me^{-\mi\kp \hx\cdot y}]^\top \int_{\mathbb{S}}\bu(y,\alpha,\sqrt{\kp/\omega}f_\fp(\alpha),\sqrt{\ks/\omega}f_\fs(\alpha)) ds_\alpha ds_y\notag \\
		=&  \frac{\me^{\mi\pi/4}}{\lambda\sqrt{8\pi\kp}}\frac{\kp^2}{\omega^2}\int_{\Gamma_D}[T_{\nu(y)}(\hx\otimes\hx)\me^{-\mi\kp \hx\cdot y}]^\top\bu_{f} ds_y,
	\end{align*}
	where  $\bu_{f}=\me^{-\mi\pi/4}\int_{\mathbb{S}} \bu(y,\alpha,\sqrt{\kp/\omega}f_\fp(\alpha),\sqrt{\ks/\omega}f_\fs(\alpha)) ds_\alpha$ is the total elastic field generated by the Herglotz wave $\uincf$. 
	
	Similarly, it holds
	\begin{eqnarray}
		f_{\fs} (\hx)\hx^\bot  =  \frac{\me^{\mi\pi/4}}{\lambda\sqrt{8\pi\ks}}\frac{\ks^2}{\omega^2}\int_{\Gamma_D}[T_{\nu(y)}(\bi-\hx\otimes\hx)\me^{-\mi\ks \hx\cdot y}]^\top \bu_{f} ds_y.
	\end{eqnarray}
	Now plugging $f$ into the definition of Herglotz wave \eqref{herglotz}, we obtain
	\begin{eqnarray}
		\uincf(x) &=& \me^{-\mi\pi/4}\int_{\ms} \left(\sqrt{\frac{\kp}{\omega}}\me^{\mi \kp \alpha\cdot x} f_{\fp}(\alpha)\alpha +\sqrt{\frac{\ks}{\omega}} \me^{\mi \ks\alpha \cdot x} f_{\fs}(\alpha)\alpha^\bot\right)ds_{\alpha}\notag \\
		&=&  \frac{1}{\lambda\sqrt{8\pi\kp}}\left(\frac{\kp}{\omega}\right)^\frac{5}{2} \int_{\Gamma_D}\left(  T_{\nu(y)}\int_{\mathbb{S}}\alpha \alpha ^\top\me^{\mi\kp \alpha \cdot (x-y)}ds_\alpha  \right)^\top\bu_{f} ds_y \notag \\
		&& +\frac{1}{\lambda\sqrt{8\pi\ks}}\left(\frac{\ks}{\omega}\right)^\frac{5}{2} \int_{\Gamma_D}\left(  T_{\nu(y)}\int_{\mathbb{S}}(\bi-\alpha \alpha ^\top )\me^{\mi\ks \alpha\cdot(x-y)}  ds_\alpha  \right)^\top\bu_{f} ds_y. \notag
	\end{eqnarray} 
	Let
	\begin{align}\label{idenEH}
		E(\kp,x,y)=\int_{\ms} \alpha\alpha^\top\me^{\mi \kp \alpha\cdot (x-y)}  ds_\alpha, \quad 
		H(\ks,x,y) = \int_{\ms} (\bi-\alpha\alpha^\top)\me^{\mi \ks \alpha\cdot (x-y)}  ds_\alpha.
	\end{align}
	According to the identity~\cite[Equation 10.9.2]{olverNISTHandbookMathematical2010}
	\begin{eqnarray}\label{idenbessel}
		J_n(z) =\frac{(-\mi)^n}{\pi}\int_0^\pi e^{\mi z \cos\theta}\cos(n\theta) d\theta,
		\quad n\in\mathbb{Z},
	\end{eqnarray}
	where $J_n$ is the Bessel function of the first kind, we can derive that
	\begin{eqnarray}\label{idenE}
		E(\kp,x,y) =\pi Q^\top\begin{bmatrix}
			J_0(z)-J_2(z) & 0 \\
			0 & J_0(z)+J_2(z)\\
		\end{bmatrix}Q, 
	\end{eqnarray}
	where $z=\kp|x-y|$ and $Q$ is an orthogonal matrix that maps $x-y$ to $[|x-y|,0]^\top$. Using \eqref{idenbessel} again with $n=0$, we also obtain
	\begin{eqnarray}\label{idenH}
		H(\ks,x,y) =2\pi J_0(\ks|x-y|)\bi- E(\ks,x,y).
	\end{eqnarray}
\end{proof}

Based on the asymptotic property of $J_0$ and $J_2$, Theorem \ref{globalf} shows the incident wave $\uincf$ generated by the eigenfunction $f$ will focus on the unknown obstacles and decay as $1/\sqrt{r}$ where $r$ is the distance from the obstacle, which is the essential property of the time reversal method. Theorem \ref{globalf} also shows one can use only one wave (either compressional or shear wave) to focus the obstacles. However, in order to obtain the eigenfunction $f$, we still need the far field pattern of both waves. Meanwhile, since the time reversal operator $T$ is self-adjoint, by min-max principle, it holds that
$$|\lambda_1|^2=\sup_{f\in \ml,||f||^2_2=1}||Ff||^2_2.$$
Therefore, the eigenfunction of the largest eigenvalue will maximize the illumination of cavities. In general, for cavities with non-negligible interactions, it is difficult to obtain the explicit form of significant eigenvalues, as well as the eigenfunctions of the far field operator $F$ with respect to the locations of cavities. However, if the cavities are small and distant, selective focusing~\cite{hazardSelectiveAcousticFocusing2004} on an individual one can be obtained when the interaction among cavities becomes weak, as shown in the following sections.

\section{Selective focusing of multiple elastic cavities }

In this section, we are concerned with the relation between the number of cavities contained in the medium and the number of significant eigenvalues of the far field operator $F$. Such a relation is usually nonlinear when the size of the cavity is  large or even comparable to the wavelength. However, when cavities are small and distant enough so that multiple scattering is negligible, we will show that that there are five significant eigenvalues associated with each cavity, and the corresponding eigenfunction will selectively focus on that cavity. To intuitively illustrate the phenomenon, we start by the analysis for a single small disk and then extend the results to general shaped cavities.

\subsection{Elastic scattering of a single discal cavity}\label{41}
In this subsection, we deal with the case where the scatterer denoted by $S_0$ is a disk at the origin  of radius $R>0$.
For this particular geometry, an explicit formula can be obtained for the eigenvalues of the far field mapping and thus for those of the time reversal operator. The results of this subsection can be seen as a natural extension from acoustic to elastic scattering based on the classical Mie theory~\cite{coltonInverseAcousticElectromagnetic2019}.

For a given point $x=(x_1,x_2)$, denote $(r,\theta)$ the polar coordinates of $x$. Let
$J_n(r)$ and $H_n^{(1)}(r)$ respectively
be the first kind Bessel and Hankel function of order $n$.
Define the scalar functions
\begin{eqnarray*}
	u_{n}^\kappa (x) = J_n(\kappa r)\me^{\mi n\theta },v_{n}^\kappa (x) = H_n^{(1)}(\kappa r)\me^{\mi n\theta },
\end{eqnarray*}
which are called \textit{cylindrical wave functions} and satisfy the two dimensional Helmholtz equation with exceptional point at the origin for $v^{\kappa}_{n}(x)$.

According to the Helmholtz decomposition \eqref{helmdec}, the incoming field for the disk $S_0$ can be expanded as
\begin{eqnarray}\label{inspan}
	\uinc(x) = \sum_{n=-\infty}^{\infty}a_{n}\nabla u_{n}^{\kp} (x)+b_{n}\nabla^\bot u_{n}^{\ks} (x),
\end{eqnarray}
where $\{a_{n},b_{n}\}$ are called the \textit{incoming expansion coefficients} of $\uinc$ on $S_0$.

Note that the shear part $\boldsymbol{u^i_\fs}$ and the compressional part $\boldsymbol{u^i_\fp}$ of an incident wave $\uinc$ in \eqref{inspan} are
\begin{eqnarray*}
	\boldsymbol{u^i_\fs}&=&-\frac{1}{\kappa_{\fs}^2}\nabla^\bot\nabla^\bot\cdot\uinc=
	\sum_{n=-\infty}^{\infty}b_{n}\nabla^\bot u_{n}^{\ks} (x),\\ \boldsymbol{u^i_\fp}&=&-\frac{1}{\kappa_{\fp}^2}\nabla\nabla\cdot\uinc= \sum_{n=-\infty}^{\infty}a_{n}\nabla u_{n}^{\kp} (x).
\end{eqnarray*}
%	Recalling the relationship between the Cartesian and polar coordinates for gradient, it holds
%	\begin{align}
	%		\nabla u_{n}^{\kp} (x)=\kp J_n'(\kp r)\me^{\mi n\theta}\hat{r}+\frac{\mi n}{r}J_n(\kp r)\me^{\mi n\theta}\hat{\theta},\\
	%		\nabla^\bot u_{n}^{\ks} (x)=\ks J_n'(\ks r)\me^{\mi n\theta}\hat{\theta}-\frac{\mi n}{r}J_n(\ks r)\me^{\mi n\theta}\hat{r}.
	%	\end{align}
For the plane wave incidence, explicit expression for these coefficients can be obtained through the vector analogue of the Jacobi–Anger expansion\cite{coltonInverseAcousticElectromagnetic2019}. 
%Detailed expression is given in the appendix \ref{appA}.
The kernel $f\in\ml$ can be expanded by
\begin{align*}
	f = (f_\fp,f_\fs)=\sum_{n=-\infty}^{\infty}(f_{n}^a\me^{\mi n\theta},f_{n}^b\me^{\mi n\theta}).
\end{align*}
Based on the plane wave expansion, the \textit{incoming expansion coefficients} for a Herglotz wave with kernel $f\in\ml$
%	 $\me^{-\mi\pi/4}(\sqrt{\kp/\omega}f_\fp(\alpha),\sqrt{\ks/\omega}f_\fs(\alpha))$ 
are simply
\begin{eqnarray}\label{hegexpan}
	a_{n}=-\frac{2\pi\mi^{n+1}\me^{-\mi\pi/4}}{\sqrt{\kp\omega}}f_n^a,\quad b_{n}=-\frac{2\pi\mi^{n+1}\me^{-\mi\pi/4}}{\sqrt{\ks\omega}}f_n^b.
\end{eqnarray}
%Since 
%\[T_r u=2\mu\frac{\partial u}{\partial r}+\hat{r}\lambda(\nabla\cdot u)-\hat{\theta}\mu\nabla^\bot\cdot u\]
%\begin{align}
%		\frac{\nabla u_{n}^{\kp} (x)}{\partial r} =\kp^2 J_n''(\kp r)\me^{\mi n\theta}\hat{r}+\mi n\frac{\kp r J_n'(\kp r)-J_n(\kp r)}{r^2}\me^{\mi n\theta}\hat{\theta}\\
%		\frac{\nabla^\bot u_{n}^{\ks} (x)}{\partial r} =\ks^2 J_n''(\ks r)\me^{\mi n\theta}\hat{\theta}-\mi n\frac{\ks r J_n'(\ks r)-J_n(\ks r)}{r^2}\me^{\mi n\theta}\hat{r}\\
%	\nabla\cdot(\nabla u_{n}^{\kp} (x))=-\kp^2 J_n(\kp r)\me^{\mi n\theta}\\
%	\nabla\cdot(\nabla^\bot u_{n}^{\ks} (x))=0\\
%	\nabla^\bot\cdot (\nabla u_{n}^{\kp} (x))=0\\
%	\nabla^\bot\cdot (\nabla^\bot u_{n}^{\ks} (x))=-\ks^2 J_n(\ks r)\me^{\mi n\theta}
%\end{align}
%We consider the traction 
%\begin{align*}
%	T_r(\uinc)|_{r=R}=\sum_{n=-\infty}^{\infty}&\bigg(a_n(2\mu\kp^2 J''_n
%	(\kp R)-\lambda\kp^2J_n
%	(\kp R))\me^{in\theta}\hat{r}\\
%	&+a_n2\mu in \frac{\kp R J'_n
	%		(\kp R)-J_n
	%		(\kp R)}{r^2}\me^{in\theta}\hat{\theta}\\
%	&-b_n2\mu in\frac{\ks R J'_n
	%		(\ks R)-J_n
	%		(\ks R)}{r^2}\me^{in\theta}\hat{r}\\
%	&+b_n(2\mu \ks^2J_n''
%	(\ks R)+\mu\ks^2J_n(\ks R))\me^{in\theta}\hat{\theta}\bigg)
%\end{align*}
After the incidence of $\uinc$, the scattered field $\bv$ in the exterior of $S_0$ is given by
\begin{eqnarray}\label{scatspan}
	\bv =\sum_{n=-\infty}^{\infty}\alpha_{n}\nabla v_{n}^{\kp} (x)+\beta_{n}\nabla^\bot v_{n}^{\ks} (x),
\end{eqnarray} 
where $\{\alpha_{n},\beta_{n}\}$ are referred as the \textit{outgoing expansion coefficients}. The linear matrix that maps all the \textit{incoming expansion coefficients} $\{a_{n},b_{n}\}$ to all the \textit{outgoing expansion coefficients} $\{\alpha_{n},\beta_{n}\}$, $n\in\mathbb{Z}$, of an elastic scatterer is referred as the scattering matrix $\mathcal{S}$~\cite{laiFastInverseElastic2022}. For the disk $S_0$, the scattering matrix $\mathcal{S}$ is block diagonal with diagonal blocks $\mathcal{S}_{n}$, of which the explicit expression is given in the appendix \ref{appA}.

Using the asymptotic property of Hankel functions~\cite{olverNISTHandbookMathematical2010},
%	\begin{align}
	%		\begin{aligned}
		%			H_n^{(1)}(z)=\sqrt{\frac{2}{\pi z}}\me^{\mi(z-\frac{n\pi}{2}-\frac{\pi}{4})}\left(1+\oo\left(\frac{1}{z}\right)\right),\quad z\rightarrow\infty,\\
		%			H_n^{(1)'}(z)=\sqrt{\frac{2}{\pi z}}\me^{\mi(z-\frac{n\pi}{2}+\frac{\pi}{4})}\left(1+\oo\left(\frac{1}{z}\right)\right),\quad z\rightarrow\infty,
		%		\end{aligned}
	%	\end{align}
the far field pattern for the scattered field $\bv$ based on the \textit{outgoing expansion coefficients} is given by
\begin{eqnarray}\label{farexpan}
	\begin{aligned}
		\left(\bv_{\fp,\infty}, \bv_{\fs,\infty}\right)= \left(\sum_{n=-\infty}^{\infty} (-\mi)^n\alpha_n\sqrt{\frac{2\kp}{\pi }}\me^{\mi\frac{\pi}{4}}\me^{\mi n\theta}\hat{r},  \sum_{n=-\infty}^{\infty} (-\mi)^n\beta_n\sqrt{\frac{2 \ks}{\pi }}\me^{\mi\frac{\pi}{4}}\me^{\mi n\theta}\hat{\theta}\right).
	\end{aligned} 
\end{eqnarray}
By using the scattering matrix $\mathcal{S}$ of the sphere $S_0$, the far field operator $F$ can also be formulated as a block diagonal matrix, where the $n$-th block $F_{n}$ is 
\begin{eqnarray}
	F_{n} = D^{scat}_{n}\mathcal{S}_{n} D^{inc}_{n}, \quad n\in\mathbb{Z}.
\end{eqnarray} 
From equations \eqref{hegexpan} and \eqref{farexpan}, $D^{inc}_{n}$ and $D^{scat}_{n}$ are $2\times2$ diagonal matrices given by 
\begin{align}
	D^{inc}_{n} =
	\begin{bmatrix}
		-\frac{2\pi\mi^{n+1}\me^{-\mi\pi/4}}{\sqrt{\kp\omega}} & 0 \\
		0 & -\frac{2\pi\mi^{n+1}\me^{-\mi\pi/4}}{\sqrt{\ks\omega}}
	\end{bmatrix},\quad
	D^{scat}_{n} =
	\begin{bmatrix}
		(-\mi)^n\sqrt{\frac{2\kp}{\pi }}\me^{\mi\frac{\pi}{4}}  & 0  \\
		0 & (-\mi)^n\sqrt{\frac{2 \ks}{\pi }}\me^{\mi\frac{\pi}{4}}
	\end{bmatrix}.
\end{align}
If the radius $R$ of the sphere $S_0$ is sufficiently small, we are able to obtain the following result.
%	\begin{theorem}\label{smalleig}
	%		When $R\rightarrow 0,|n|\geq 2$, the two eigenvalues of $F_{n}$, given by $\lambda^{i}_{n}, i = 1,2$, satisfy
	%		\begin{eqnarray}\label{eigexp}
		%			\begin{split}
			%				\lambda^{1}_{n} &= \frac{-8c_nd_n\mi |n|(|n|+1)(|n|-1)\mu^2(\kp^{|n|}+\ks^{|n|})}{\pi\kp^{|n|}\ks^{|n|}}R^{2|n|-2}+\mathcal{O}(R^{2|n|}),\\
			%				\lambda^{2}_{n} &= \frac{2c_nd_n\mi\mu^2|n|^2\kp^{|n|}\ks^{|n|}}{\pi}R^{2|n|}+\mathcal{O}(R^{2|n|+2}),
			%			\end{split}
		%		\end{eqnarray}
	%		where 
	%		$c_n = \frac{\pi^2\kp^{|n|}\ks^{|n|} }{2^{2{|n|}+1}\mu(\ks^2\mu {|n|}+\kp^2(\lambda+\mu-\lambda {|n|}))\Gamma({|n|}-1)\Gamma({|n|}+2)},d_n=-\sqrt{\frac{8\pi}{\omega}}\mi$. When $n=0$, $\lambda_0^1=\oo(R^2),\lambda_0^2=\oo(R^4)$, when $|n|=1$, $\lambda_1^1=\lambda_{-1}^1=\oo(R^2),\lambda_1^2=\lambda_{-1}^2=\oo(R^4)$. The eigenfunctions for $\lambda^1_{n}$ and $\lambda^2_{n}$ lie in the space of $\{\me^{\mi n\theta}\hat{r},\me^{\mi n\theta}\hat{\theta}\}$ for all $n\in\mathbb{Z}$.
	%	\end{theorem}
%	
\begin{theorem}\label{smalleig}
	When $R\rightarrow 0,n\geq 2$, the two eigenvalues of $F_{n}$, given by $\lambda^{i}_{n}, i = 1,2$, satisfy
	\begin{eqnarray}\label{eigexp}
		\begin{aligned}
			\lambda^{1}_{n} &= \frac{-8c_nd_n\mi n(n+1)(n-1)\mu^2(\kp^{n}+\ks^{n})}{\pi\kp^{n}\ks^{n}}R^{2n-2}+\mathcal{O}(R^{2n}),\\
			\lambda^{2}_{n} &= \frac{2c_nd_n\mi\mu^2n^2\kp^{n}\ks^{n}}{\pi}R^{2n}+\mathcal{O}(R^{2n+2}),
		\end{aligned}
	\end{eqnarray}
	where 
	$c_n = \frac{\pi^2\kp^{n}\ks^{n} }{2^{2{n}+1}\mu(\ks^2\mu {n}+\kp^2(\lambda+\mu-\lambda {n}))\Gamma({n}-1)\Gamma({n}+2)},d_n=-\sqrt{\frac{8\pi}{\omega}}\mi$. When $n=0,1$, it holds $\lambda_n^1=\oo(R^2),\lambda_n^2=\oo(R^4)$, and when $n<0$, $\lambda_n^i=\lambda_{-n}^i,i=1,2$. The eigenfunctions for $\lambda^1_{n}$ and $\lambda^2_{n}$ lie in the space of $\{\me^{\mi n\theta}\hat{r},\me^{\mi n\theta}\hat{\theta}\}$ for all $n\in\mathbb{Z}$.
\end{theorem}

Proof of Theorem \ref{smalleig} is simply based on the explicit expression of $F_{n}$ and the asymptotic expansions of Bessel functions~\cite{olverNISTHandbookMathematical2010}
\begin{eqnarray}\label{asybessel}
	\begin{aligned}
		J_n(z) &= 1/(2^n\Gamma(n+1))\left(z^n-z^{n+2}/(4n+4)\right)+\oo(z^{n+4}), \\
		H_n^{(1)}(z) & =  -\mi 2^n\Gamma(n)/{\pi}(z^{-n}+z^{-(n-2)}/(4n-4)) + \oo(z^{-n+4}),
	\end{aligned}
	\mbox{ for } z\rightarrow 0.
\end{eqnarray}
From Theorem \ref{smalleig}, one can see that the ratio  $\lambda^i_{|n|+1}/\lambda^i_{|n|}=\mathcal{O}(R^2)$ for i=1,2 and $|n|\geq 2$, which is similar to the behavior of elastic scattering of rigid spheres~\cite{laiFastInverseElastic2022}. In particular, for a discal cavity with radius $R\rightarrow 0$, there are five significant eigenvalues, namely, $\lambda^1_{-2}$, $\lambda^1_{-1}$, $\lambda^1_{0}$, $\lambda^1_{1}$ and $\lambda^1_{2}$, as their magnitude is on the order of $\mathcal{O}(R^2)$ (others are on the order of $\mathcal{O}(R^4)$ or higher) and the corresponding eigenfunctions dominate the far field scattering. Such a conclusion can be extended to the general-shaped small cavities, as shown in the next subsection. 
%	\subsection{Asymptotic expansion of the far field for small particles}
\subsection{Elastic scattering of multiple small cavities}
In this subsection, based on the single and double layer operators, we show that the effect of multiple scattering can be neglected when the cavities are far from each other, in which case the inverse scattering of multiple cavities is essentially reduced to the reconstruction of a single cavity. Then we derive the asymptotic far field expression of a single small cavity in the case of longitudinal and transversal plane wave incidence. Finally the limit far field operator of several small and well-resolved cavities is presented.

Given $M$ cavities $D_l,1\leq l\leq M$, each is bounded and simply connected in $\mr$ with boundary $\Gamma_l$. In order to simplify the exposition, we shall assume that $k=(\kp,\ks)\sim 1$. Hence, the size of a scatterer $D_l$ can be interpreted in terms of its Euclidean diameter.  We denote $\Phi_k(x,y)$ the fundamental solution \eqref{greenfun} that depends on $k$. Define the single and double layer potential operators, respectively, by
\begin{align}
	(S_l\varphi)(x) & :=2\int_{\Gamma_l}\Phi_k(x,y)\varphi(y)ds_y,\quad\varphi\in
	C(\Gamma_l),\quad x\in\Gamma_l\label{sp},                                                                \\
	(K_l\varphi)(x) & :=2\int_{\Gamma_l}\Big[T_{\nu(y)}\Phi_k(x,y)\Big]^\top\varphi(y)ds_y,\quad\varphi\in
	C(\Gamma_l),\quad x\in\Gamma_l\label{dp}.
\end{align}
The adjoint operator $K_l'$ of $K_l$ is given by
\begin{align}
	(K_l'\varphi)(x):=2\int_{\Gamma_l}\Big[T_{\nu(x)}\Phi_k(x,y)\Big]^\top\varphi(y)ds_y,\quad\varphi\in
	C(\Gamma_l),\quad x\in\Gamma_l\label{asp}.
\end{align}
Recalling that $H_n^{(1)}=J_n+\mi Y_n$, where $J_n$ and $Y_n$ are the Bessel function and Neumann function of order n, respectively, we find from the series expansions of these functions~\cite{coltonInverseAcousticElectromagnetic2019}
that as $k|x-y|\rightarrow 0$,
\begin{align}
	\Phi_k(x,y) = &\Phi_0(x,y)+\frac{\mi}{4\mu}
	\left(C+\frac{2\mi}{\pi}\ln \ks\right) \bi
	+\frac{\mi}{4}\left(C_2\left(\frac{\ln (\kp)}{\lambda+2\mu}-\frac{\ln (\ks)}{\mu}\right)\right.+C_3\left.\left(\frac{1}{\lambda+2\mu}-\frac{1}{\mu}\right)\right)\bi\label{rel}\\&+\mathcal{O}({k^2|x-y|^2\ln(k|x-y|)}),\quad x\neq y.\notag \\
	\frac{1}{k}\sigma(\Phi_k&(x,y)\boldsymbol{e_j})=\frac{1}{k}\sigma(\Phi_0(x,y)\boldsymbol{e_j})+\mathcal{O}({k|x-y|\ln(k|x-y|)}),\quad j=1,2, x\neq y.\label{relder}
\end{align}
where $$\Phi_0(x,y):=-\frac{\lambda+3\mu}{4\mu\pi(\lambda+2\mu)}\ln(|x-y|)\bi
+\frac{\lambda+\mu}{4\pi\mu(\lambda+2\mu)}\frac{(y-x)\otimes(y-x)}{|x-y|^2}$$ denotes the fundamental solution for the Lam\'{e} system and $$C=1+\frac{2\mi}{\pi}(\gamma-\ln 2),C_1=-\frac{2\mi}{\pi},
C_2=\frac{\mi}{\pi},C_3=\frac{1}{2}+\frac{2 \gamma -1-2 \ln (2)}{2 \pi }\mi,$$ are  constants. 

Similar to the definitions of $S_l, K_l, K_l'$, we define the operators $\widetilde{S}_l, \widetilde{K}_l, \widetilde{K}_l'$ in the same way as \eqref{sp},\eqref{dp}
and \eqref{asp}, but with $\Phi_k(x,y)$ replaced by $\Phi_0(x,y)$.

First, we consider the scattering from multiple sparsely distributed cavities. The following theorem show that the effect of multiple scattering can be neglected when they are sufficiently distant.
Set
\begin{align}\label{obs}
	D = \bigcup\limits_{l=1}^M D_l\quad \textrm{and}\quad L=\min_{l
		\neq l',1\leq l,l'\leq M} \textrm{dist}(\overline{D_l},\overline{D_{l'}}).
\end{align}
\begin{theorem}\label{negmul}
	Consider elastic scattering of multiple cavities given in \eqref{obs}, for $L$ sufficiently large, we have
	\begin{align}\label{mulsca}
		\bv_\infty(\hat{x};D)=\sum_{l=1}^{M}\bv_\infty(\hat{x};D_l)+\mathcal{O}\left(\frac{1}{\sqrt{L}}\right).
	\end{align}
\end{theorem}
\begin{proof}
	For simplicity we first assume that $M = 2$ and the pair $(\kp,\ks)$ is not an eigenvalue of the elastic scattering problem in $D_l$ associated with the homogeneous traction-free boundary condition on $\Gamma_l,l=1,2$.
	
	The scattered field $\bv(x;D_l)$ corresponding to $D_l$ can be represented as the single layer potential
	\[\bv(\hat{x};D_l)=\int_{\Gamma_l}\Phi(x,y)\phi_l(y)d s_y,\quad x\in\mr\setminus\overline{D_l},\]
	where the density function $\phi_l\in (C(\Gamma_l))^2$ is uniquely determined from the traction-free boundary condition on $\Gamma_l$, and is implied in the boundary integral equation
	\[\phi_l=2(I-K_l')^{-1}(T_\nu u^{in}|_{\Gamma_l}),\quad l=1,2.\]
	Note that here $I$ denotes the identity operator. The uniqueness and existence of $\phi_l$ follow from the Fredholm alternative applied to the operator $I-K_l'$. To prove the theorem for the scatterer $
	D=D_1\cup D_2$, we make use of the ansatz
	\[\bv(x;D)=\sum_{l=1,2}\left\{\int_{\Gamma_l}\Phi(x,y)\varphi_l(y)d s_y\right\},\quad x\in\mr\setminus\overline{D}\]
	with $\varphi_l\in C(\Gamma_l)$. By using the boundary condition
	$T_\nu(\uinc+\bv)=0$ on each $\Gamma_l$, we obtain the system of integral equations
	\begin{align}\label{ineqs}
		\begin{pmatrix}
			I-K_1' & J_2    \\
			J_1    & I-K_2'
		\end{pmatrix}
		\begin{pmatrix}
			\varphi_1 \\
			\varphi_2
		\end{pmatrix}=2\begin{pmatrix}
			T_{\nu}\uinc|_{\Gamma_1} \\T_{\nu}\uinc|_{\Gamma_2}
		\end{pmatrix},
	\end{align}
	where  the operators $J_1:C(\Gamma_1)\rightarrow C(\Gamma_2),
	J_2:C(\Gamma_2)\rightarrow C(\Gamma_1)$ are defined respectively by
	\[(J_1\varphi_1)(x):=-2\int_{\Gamma_1}[T_{\nu(x)}\Phi(x,y)]\varphi_1(y)d s_y,\quad x\in\Gamma_2,\]
	\[(J_2\varphi_2)(x):=-2\int_{\Gamma_2}[T_{\nu(x)}\Phi(x,y)]\varphi_2(y)d s_y,\quad x\in\Gamma_1.\]
	Since $L\gg 1$, using the asymptotic behavior of the gradient of the fundamental solution \eqref{greenfun}, one readily estimates
	\[||J_1\varphi_1||_{C(\Gamma_1)}\leq C_1L^{-\frac{1}{2}}||\varphi_1||_{C(\Gamma_1)},||J_2\varphi_2||_{C(\Gamma_2)}\leq C_2L^{-\frac{1}{2}}||\varphi_2||_{C(\Gamma_2)},\quad C_1,C_2>0.\]
	Hence, it follows from \eqref{ineqs} and the invertibility of $I-K_l':C(\Gamma_l)\rightarrow C(\Gamma_l)$ that
	\begin{align}
		\begin{pmatrix}
			\varphi_1 \\\varphi_2
		\end{pmatrix}=		\begin{pmatrix}
			(I-K_1')^{-1} & 0             \\
			0             & (I-K_2')^{-1}
		\end{pmatrix}
		\begin{pmatrix}
			2T_{\nu}\uinc|_{\Gamma_1} \\2T_{\nu}\uinc|_{\Gamma_2}
		\end{pmatrix}+\oo\left(\frac{1}{\sqrt{L}}\right)=\begin{pmatrix}
			\phi_1 \\\phi_2
		\end{pmatrix}+\oo\left(\frac{1}{\sqrt{L}}\right).
	\end{align}
	This implies that
	\[\bv(x;D)=\bv(x;D_1)+\bv(x;D_2)+\oo\left(\frac{1}{\sqrt{L}}\right),\]
	which further leads to
	\[\bv_{\infty}(x;D)=\bv_{\infty}(x;D_1)+\bv_{\infty}(x;D_2)+\oo\left(\frac{1}{\sqrt{L}}\right).\]
	The case that $D$ has more than two components can be proved in a similar manner by making use of the integral equation method. In the argument above, there is a technical assumption that $(\kp,\ks)$ is not an eigenvalue for the elastic scattering problem in $D_l$ with the traction-free boundary condition. If the eigenvalue problem happens, one can make use of the combined layer potentials~\cite{hsiaoBoundaryIntegralEquations2008} and then by a completely similar argument as above, one can show \eqref{mulsca}.
	Hence the proof is completed.
\end{proof}

Next, we study elastic scattering from multiple small cavities based on the above “independent scattering”. We assume that $D_l,1\leq l\leq M$ contains the origin and its diameter is comparable with the wavelength, i.e., $\textrm{diam}(D_l)\sim\mathcal{O}(1)$. For $\rho\in \mathbb{R}_+$, we introduce a dilation operator $\Lambda_\rho$ by
\[\Lambda_\rho D_l:=\{\rho x|x\in D_l\}\]
and set
\[D_l^\rho:=s_l+\Lambda_\rho D_l,\quad s_l\in\mr,1\leq l\leq M.\]
Let
\begin{align}\label{mulobs}
	D^\rho:=\bigcup\limits_{l=1}^MD_l^\rho.
\end{align}
We have the following theorem which gives the asymptotic far field pattern similar to the three-dimensional case in \cite{huRecoveringComplexElastic2014}.
\begin{theorem}\label{asypd}
	Consider multiple elastic cavities $D^\rho$ given in \eqref{mulobs}. Assume that
	$\rho\ll 1,k\sim 1$ and
	\begin{align}
		L=\min_{l\neq l',1\leq l,l'\leq M} {\rm dist} (s_l,s_{l'})\gg 1.
	\end{align}
	Then the asymptotic compressional and shear far field pattern excited by longitudinal plane wave $\boldsymbol{u^i_{\fp}}=d\me^{\mi k x\cdot d}$ and transversal plane wave $\boldsymbol{u^i_{\fs}}=d^\bot\me^{\mi k x\cdot d}$, denoted by $\bv_{\fp\fp}^\infty,\bv_{\fs\fp}^\infty,\bv_{\fp\fs}^\infty,\bv_{\fp\fs}^\infty$, are given as follows,
	\begin{align*}
		\bv_{\fp\fp}^\infty(\hat{x},D^\rho)=&-\rho^2\gamma\kp^{\frac{3}{2}}\bigg((\hat{x}\otimes\hat{x})\sum_{l=1}^M\me^{\mi s_l\cdot(\kp d-\kp\hat{x})}
		\left[\omega^2d|D_l|+2\omega^2(\hat{x}\cdot\mathbb{P}_l)\cdot\mathbb{L}(\lambda,\mu,d)\right]\\
		&+\oo(\rho\ln \rho+L^{-\frac{1}{2}})\bigg),\\
		\bv_{\fs\fp}^\infty(\hat{x},D^\rho)=&-\rho^2\gamma\ks^{\frac{3}{2}}\bigg((\bi-(\hat{x}\otimes\hat{x}))\sum_{l=1}^M\me^{\mi s_l\cdot(\kp d-\ks\hat{x})}
		\big[\omega^2d|D_l|+2\kp\ks(\lambda+2\mu)(\hat{x}\cdot\mathbb{P}_l)\cdot\mathbb{L}(\lambda,\mu,d)\big]\\&+\oo(\rho\ln \rho+L^{-\frac{1}{2}})\bigg), \\
		\bv_{\fp\fs}^\infty(\hat{x},D^\rho)=&-\rho^2\gamma\kp^{\frac{3}{2}}\bigg((\hat{x}\otimes\hat{x})\sum_{l=1}^M\me^{\mi s_l\cdot(\ks d-\kp\hat{x})}
		\left[\omega^2d^\perp|D_l|+2\kp\ks\mu(\hat{x}\cdot\mathbb{P}_l)\cdot\mathbb{H}(d)\right]\\&+\oo(\rho\ln \rho+L^{-\frac{1}{2}})\bigg),\\
		\bv_{\fs\fs}^\infty(\hat{x},D^\rho)
		=&-\rho^2\gamma\ks^{\frac{3}{2}}\bigg((\bi-(\hat{x}\otimes\hat{x}))\me^{\mi s_l\cdot(\ks d-\ks\hat{x})}
		\left[\omega^2d^\perp|D_l|+2\omega^2(\hat{x}\cdot\mathbb{P}_l)\cdot\mathbb{H}(d)\right]\\&+\oo(\rho\ln \rho+L^{-\frac{1}{2}})\bigg).
	\end{align*}
	Here, $\gamma= \frac{\me^{\mi\pi/4}}{\sqrt{8k\pi }}$, $|D_l|$ denotes the surface area of $D_l$ and $\mathbb{P}_l$ denotes the polarization tensor corresponding to the cavity $D_l$ as introduced in \eqref{tens}. $\mathbb{L}$ and $\mathbb{H}$ are defined by
	\eqref{del}, \eqref{deh} respectively.
\end{theorem}
\begin{proof}
	By Theorem \ref{negmul}, it suffices to analyze the asymptotics of the far-field patterns for only one single cavity. For notational convenience, we let $\Omega=D_l^\rho$ and $\Sigma=D_l$ for any fixed $1\leq l\leq M$, so that $\Omega = z+\Lambda_\rho\Sigma$. For
	$f\in C(\partial \Omega)$ and $g\in C(\partial {\Sigma})$, we introduce the transforms
	\begin{align*}
		\hat{f}(\xi)=f^\wedge:=f(\rho\xi+z),\xi\in\partial {\Sigma},\quad
		\check{g}(x)=g^\vee:=g((x-z)/\rho),x\in\partial \Omega.
	\end{align*}
	Using change of variables it is not difficult to verify that~\cite{griesmaierMultifrequencyOrthogonalitySampling2011}
	\begin{align*}
		K_\Omega\phi=(K_{\Sigma}\hat{\phi})^\vee,\quad(I-K_\Omega)\phi=((I-K_{\Sigma})\hat{\phi})^\vee,\quad(I-K_\Omega)^{-1}\phi=((I-K_{\Sigma})^{-1}\hat{\phi})^\vee,
	\end{align*}
	and	similarly
	\begin{align*}
		K'_\Omega\phi=(K'_{\Sigma}\hat{\phi})^\vee,\quad(I-K'_\Omega)\phi=((I-K'_{\Sigma})\hat{\phi})^\vee,\quad(I-K'_\Omega)^{-1}\phi=((I-K'_{\Sigma})^{-1}\hat{\phi})^\vee.
	\end{align*}
	These identities also hold for $\widetilde{K}, \widetilde{K}'$ defined before via
	$\widetilde{\Phi}(x,y)=\Phi_0(x,y)$. Hence using \eqref{relder}, there holds
	\begin{align*}
		(I-K'_\Omega)\phi-((I-\widetilde{K}'_{\Sigma})\hat{\phi})^\vee
		& =(I-K'_\Omega)\phi-(I-\widetilde{K}'_\Omega)\phi
		\\
		& =(\widetilde{K}'_\Omega-K'_\Omega)\phi
		\\
		& =2\int_{\partial\Omega}\nu(x)\cdot[\sigma(\Phi_k(x,y)-\Phi_0(x,y))]\phi(y)d s_y
		\\
		& =2\int_{\partial {\Sigma}}\nu(\xi)\cdot(\oo(\rho\ln\rho))\hat{\phi}(\eta)\rho d s_\eta\sim \mathcal{O}(\rho^2\ln \rho)
	\end{align*}
	as $\rho\rightarrow 0$. Since $\rho\ll 1$, by Neumann series we have
	\begin{align}\label{asyest}
		(I-K'_\Omega)^{-1}\phi=((I-\widetilde{K}'_{\Sigma})^{-1}\hat{\phi})^\vee+\oo(\rho^2\ln \rho),\quad\rho\rightarrow 0.
	\end{align}
	To proceed with the proof, we represent the scattered field $\bv(x;\Omega)$ as the single layer potential
	\begin{align*}
		\bv(x;\Omega)=\int_{\partial\Omega}\Phi(x,y)\varphi(y)d s_y,
		\quad x\in\mr\setminus\overline{\Omega},
	\end{align*}
	with the density function $\varphi\in (C(\partial \Omega))^2$ given by
	\begin{align*}
		\varphi=2(I-K'_\Omega)^{-1}(T_\nu \uinc|_{\partial\Omega}).
	\end{align*}
	Then, using the asymptotic behavior of the fundamental solution \eqref{greenfun}, the compressional part $\bv_{\fp,\infty}$ and shear part $\bv_{\fs,\infty}$ of the far field $\bv_\infty$ are given by
	\begin{align}\label{far}
		\bv_{\fp,\infty}(\hat{x};\Omega) & =2\gamma\kp^{\frac{3}{2}}(\hat{x}\otimes\hat{x})\int_{\partial\Omega}\me^{-\mi \kp\hat{x}\cdot y}
		[(I-K'_\Omega)^{-1}\psi](y)d s_y,\\
		\bv_{\fs,\infty}(\hat{x};\Omega) & =2\gamma\ks^{\frac{3}{2}}(\bi-\hat{x}\otimes\hat{x})\int_{\partial\Omega}\me^{-\mi \ks\hat{x}\cdot y}
		[(I-K'_\Omega)^{-1}\psi](y)d s_y,
	\end{align}
	where $\gamma = \frac{\me^{\mi\pi/4}}{\sqrt{8\pi}{\omega^2}}, \psi:=T_\nu \uinc|_{\partial\Omega}=\nu\cdot\sigma(\uinc)|_{\partial\Omega}$.
	Changing the variable $y=z+\rho \xi$ with $\xi\in\partial {\Sigma}$ in \eqref{far} and
	making use of the estimate \eqref{asyest}, we find
	\begin{align}\label{asyfar}
		\bv_{\fp,\infty}(\hat{x};\Omega) & =2\rho\gamma\kp^{\frac{3}{2}}(\hat{x}\otimes\hat{x})\int_{\partial {\Sigma}}\me^{-\mi \kp\hat{x}\cdot (z+\rho\xi)}
		[(I-\widetilde{K}_{\Sigma}')^{-1}\hat{\psi}(\xi)+\oo(\rho^2\ln \rho)]d s_\xi.
	\end{align}
	Expanding the exponential function $\xi\rightarrow\me^{-\mi \kp\hat{x}\cdot(z+\rho\xi)}$ around $z$ in terms of $\rho$ yields
	\begin{align}\label{incexp}
		\me^{-\mi \kp\hat{x}\cdot(z+\rho\xi)}=\me^{-\mi \kp\hat{x}\cdot z}-\mi \kp\rho(\hat{x}\cdot\xi)\me^{-\mi \kp\hat{x}\cdot z}+\oo(\rho^2),\quad \rho\rightarrow0.
	\end{align}
	Inserting \eqref{incexp} into \eqref{asyfar} gives
	\begin{align}\label{asyfar2}
		\begin{aligned}
			\bv_{\fp,\infty}(\hat{x};\Omega)=
			& 2\rho\gamma\kp^{\frac{3}{2}}\me^{-\mi \kp\hat{x}\cdot z}(\hat{x}\otimes\hat{x})\left(\int_{\partial {\Sigma}}
			(I-\widetilde{K}'_{\Sigma})^{-1}\hat{\psi}(\xi)d s_\xi\right)
			\\
			& -2\mi\rho^2\gamma \kp^{\frac{5}{2}}\me^{-\mi\kp\hat{x}\cdot z}(\hat{x}\otimes\hat{x})\left(\int_{\partial {\Sigma}}(\hat{x}\cdot\xi)
			(I-\widetilde{K}'_{\Sigma})^{-1}\hat{\psi}(\xi)d s_\xi\right)+\oo(\rho^3\ln \rho).
		\end{aligned}
	\end{align}
	To estimate the integrals on the right hand side of \eqref{asyfar2}, we will investigate the  longitudinal and transversal incident plane waves, respectively.\\
	\textbf{Case \Rmnum{1}}:\quad$\uinc=\boldsymbol{u^i_{\fp}}=d\me^{\mi\kp x\cdot d}$
	%		
	%		investigate the asymptotics for plane wave $\uinc=\me^{\mi kx\cdot d}$. Expanding the function
	%		$\xi\rightarrow(\nabla u^{i})^\wedge(\xi)=\nabla \uinc(\rho\xi+z)$
	%		\begin{align}\label{normexp}
		%			(\nabla u^{i})^\wedge(\xi)=\mi kd\me^{\mi k z\cdot d}
		%			[ 1+\mi k(d \cdot\xi)\rho+\oo(\rho^2) ],\quad\rho\rightarrow 0
		%		\end{align}
	
	A straightforward insertion of $\boldsymbol{u^i_{\fp}}$ into \eqref{sigmau} shows
	\begin{align}\label{sigmap}
		(\sigma(\boldsymbol{u^i_{\fp}})^\wedge)(\xi)=\mi(\lambda+2\mu)\kp\me^{\mi\kp z\cdot d}
		\mathbb{L}(\lambda,\mu,d)\left[1+\mi\kp\rho(d\cdot\xi)+\oo(\kp^2\rho^2)\right],
	\end{align}
	where 
	\begin{align}\label{del}
		\mathbb{L}(\lambda,\mu,d):=\frac{(\lambda I+2\mu(d\otimes d))}{\lambda+2\mu}.
	\end{align}
	Recalling that~\cite{kressLinearIntegralEquations2014}
	\begin{align*}
		\widetilde{K}_{\Sigma}1=2\int_{\partial {\Sigma}}\frac{\partial \Phi_0(x,y)}{\partial\nu(y)}d s_y=-1, \quad x\in \partial \Sigma,
	\end{align*}
	we see $(I-\widetilde{K}_{\Sigma})^{-1}1=\frac{1}{2}$, and thus
	\begin{align}\label{firstterm}
		\begin{aligned}
			\int_{\partial {\Sigma}}(I-\widetilde{K}'_{\Sigma})^{-1}\hat{\psi}(\xi)d s_\xi=\int_{\partial {\Sigma}}\hat{\psi}(\xi)(I-\widetilde{K}_{\Sigma})^{-1}1d s_\xi
			=\frac{1}{2}\int_{\partial {\Sigma}}\nu(\xi)\cdot(\sigma(\boldsymbol{u^i_{\fp}})^\wedge)(\xi)d s_\xi.
		\end{aligned}
	\end{align}
	Inserting \eqref{sigmap} into \eqref{firstterm} and applying Gauss’s theorem yield
	\begin{align}\label{term3}
		\begin{aligned}
			\int_{\partial {\Sigma}}(I-\widetilde{K}'_{\Sigma})^{-1}\hat{\psi}(\xi)d s_\xi
			=&\frac{\mi(\lambda+2\mu)\kp\me^{\mi\kp z\cdot d}}{2}
			\left\{\int_{\Sigma}\mathrm{div}_\xi[\mathbb{L}(\lambda,\mu,d)(1+\mi\kp\rho d\cdot \xi)]d s_\xi\right\}+\oo(\rho^2)\\
			=&-\frac{\rho\omega^2 d\me^{\mi\kp z\cdot d}|{\Sigma}|}{2}+\oo(\rho^2).
		\end{aligned}
	\end{align}
	Note that $|{\Sigma}|$ denotes the surface area of ${\Sigma}$. Again using \eqref{sigmap} we can evaluate the second integral over $\partial {\Sigma}$ 
	on the right hand of \eqref{asyfar2} as follows
	\begin{align}\label{term4}
		\begin{aligned}
			\int_{\partial {\Sigma}}(\hat{x}\cdot\xi)(I-\widetilde{K}'_{\Sigma})^{-1}\hat{\psi}(\xi)d s_\xi
			=-\mi(\lambda+2\mu)\kp\me^{\mi\kp z\cdot d}(\hat{x}\cdot\mathbb{P})\cdot\mathbb{L}(\lambda,\mu,d)+\oo(\rho),
		\end{aligned}
	\end{align}
	where the polarization tensor $\mathbb{P}$ depending only on $\Sigma$ is defined as
	\begin{align}\label{tens}
		\mathbb{P}=-\int_{\partial {\Sigma}}\xi\otimes(I-\widetilde{K}'_{\Sigma})^{-1}\nu(\xi)
		d s_\xi.
	\end{align}
	Now, combining \eqref{asyfar2},\eqref{term3} and \eqref{term4} gives the asymptotics
	\begin{align}\label{farpp}
		\begin{aligned}
			\bv_{\fp\fp}^\infty(\hat{x},\Omega)
			=&-\rho^2\kp^{\frac{3}{2}}\gamma(\hat{x}\otimes\hat{x})\me^{\mi z\cdot(\kp d-\kp\hat{x})}
			\left[\omega^2d|{\Sigma}|+2\kp^2(\lambda+2\mu)(\hat{x}\cdot\mathbb{P})\cdot\mathbb{L}(\lambda,\mu,d)\right]\\&+\oo(\rho^3\ln \rho).
		\end{aligned}
	\end{align}
	Similarly,
	\begin{align}\label{farsp}
		\begin{aligned}
			\bv_{\fs\fp}^\infty(\hat{x},\Omega)
			=&-\rho^2\ks^{\frac{3}{2}}\gamma(I-(\hat{x}\otimes\hat{x}))\me^{iz\cdot(\kp d-\ks\hat{x})}
			\big[\omega^2d|\Sigma|+2\kp\ks(\lambda+2\mu)(\hat{x}\cdot\mathbb{P})\cdot\mathbb{L}(\lambda,\mu,d)\big]\\&+\oo(\rho^3\ln \rho).
		\end{aligned}
	\end{align}
	\noindent	\textbf{Case \Rmnum{2}}:\quad$\uinc=\boldsymbol{u^i_{\fs}}=d^\bot\me^{\mi\ks x\cdot d}$
	
	Similarly, a straightforward insertion of $\boldsymbol{u^i_{\fs}}$ into \eqref{sigmau} shows that 
	\begin{align}\label{tractionexp}
		(\sigma(\boldsymbol{u^i_{\fs}})^\wedge)(\xi)=\mi\mu\ks\me^{\mi\ks z\cdot d}\mathbb{H}(d)
		[1+\mi\ks\rho(d\cdot\xi)+\oo(\ks^2\rho^2)],
	\end{align}
	where 
	\begin{align}\label{deh}
		\mathbb{H}:=(d^\perp\otimes d)+(d^\perp\otimes d)^\top.
	\end{align} 
	As a consequence of
	\eqref{firstterm}, we have $\psi=\nu\cdot\sigma(\boldsymbol{u^i_{\fs}})|_{\partial\Omega}$.
	Thus,
	\begin{align}\label{term1}
		\begin{aligned}
			\int_{\partial {\Sigma}}(I-\widetilde{K}'_{\Sigma})^{-1}\hat{\psi}(\xi)d s_\xi
			=&\frac{\mi\mu\ks\me^{\mi\ks z\cdot d}}{2}
			\left(\int_{\Sigma}\mathrm{div}_\xi(\mathbb{H}(d)(1+\mi\ks\rho d\cdot \xi))d s_\xi\right)+\oo(\rho^2)\\
			=&-\frac{\rho\omega^2 d^\perp\me^{\mi\ks z\cdot d}|{\Sigma}|}{2}+\oo(\rho^2).
		\end{aligned}
	\end{align}
	Similar to \eqref{term4}, one has
	\begin{align}\label{term2}
		\begin{aligned}
			\int_{\partial {\Sigma}}(\hat{x}\cdot\xi)(I-\widetilde{K}'_{\Sigma})^{-1}\hat{\psi}(\xi)d s_\xi
			=&\mi\mu\ks\me^{\mi\ks z\cdot d}\left(\int_{\partial {\Sigma}}(\hat{x}\cdot\xi)
			(I-\widetilde{K}'_{\Sigma})^{-1}(\nu(\xi)\cdot\mathbb{H}(d))d s_\xi\right)+\oo(\rho)\\
			=&-\mi\mu\ks\me^{\mi\ks z\cdot d}(\hat{x}\cdot\mathbb{P})\cdot\mathbb{H}(d)+\oo(\rho),
		\end{aligned}
	\end{align}
	where the polarization tensor $\mathbb{P}$ is given as the same in \eqref{tens}.
	Therefore, the insertion of \eqref{term1} and\eqref{term2} into \eqref{asyfar2}
	yields
	\begin{align}\label{farps}
		\begin{aligned}
			\bv_{\fp\fs}^\infty(\hat{x},\Omega)=&2\rho\gamma\kp^{\frac{3}{2}}(\hat{x}\otimes \hat{x})\me^{-\mi\kp\hat{x}\cdot z}\left(\int_{\partial\Omega}
			(I-\widetilde{K}'_{\Sigma})^{-1}\hat{\varphi}(\xi)d s_\xi\right)\\
			&-2i\rho^2\gamma\kp^{\frac{5}{2}}(\hat{x}\otimes \hat{x})\me^{-\mi\kp\hat{x}\cdot z}\left(\int_{\partial\Omega}(\hat{x}\cdot\xi)
			(I-\widetilde{K}'_{\Sigma})^{-1}\hat{\varphi}(\xi)d s_\xi\right)
			+\oo(\rho^3\ln \rho)\\
			=&-\rho^2\kp^{\frac{3}{2}}\gamma(\hat{x}\otimes\hat{x})\me^{\mi z\cdot(\ks d-\kp\hat{x})}
			\left[\omega^2d^\perp|{\Sigma}|+2\kp\ks\mu(\hat{x}\cdot\mathbb{P})\cdot\mathbb{H}(d)\right]+\oo(\rho^3\ln \rho).
		\end{aligned}
	\end{align}
	Similarly,
	\begin{align}\label{farss}
		\begin{aligned}
			\bv_{\fs\fs}^\infty(\hat{x},\Omega)=&-\rho^2\ks^{\frac{3}{2}}\gamma(\bi-(\hat{x}\otimes\hat{x}))\me^{\mi z\cdot(\ks d-\ks\hat{x})}
			\left[\omega^2d^\perp|\Sigma|+2\ks^2\mu(\hat{x}\cdot\mathbb{P})\cdot\mathbb{H}(d)\right]+\oo(\rho^3\ln \rho).
		\end{aligned}
	\end{align}
	Combining \eqref{farpp},\eqref{farsp},\eqref{farps},\eqref{farss} and Theorem \ref{negmul}, the proof is completed.
\end{proof}

\subsection{Selective focusing using DORT method}\label{selsec}
Throughout this section, we focus on the regime $\rho\ll k^{-1}\ll L$. Based on the asymptotic time harmonic far field model, we first derive the explicit expression of the limit far field operator $F^0$ which approximates $F^\rho$. Then compute the approximate eigenvalues and eigenvectors of $F^0$. In the end, we prove that these eigenfunctions selectively focus on the cavities based on the decaying property of oscillatory integrals.

According to Theorem \ref{asypd}, neglecting the remainder terms and by linearity, the limit operator $F^0:\ml\rightarrow\ml$, which can approximate the far field operator $F^\rho$, is defined by
%We denote by $F^0:\ml\rightarrow \ml$ the limit far field operator explicitly given by
\begin{align}\label{limoper}
	F^0f(\hat{x})=\sum_{l=1}^{M}(g_{l,\fp},g_{l,\fs}),
\end{align}
where
\begin{align}
	g_{l,\fp}(\hat{x})= &\kp^{\frac{3}{2}}\me^{-\mi\kp\hat{x}\cdot s_l}\hat{x}\cdot\bigg[\omega^2\int_\ms\me^{\mi\kp \alpha\cdot s_l}f_\fp(\alpha)\alpha
	d s_\alpha |D_l|+\frac{2\lambda\omega^2}{\lambda+2\mu}\int_\ms\me^{\mi\kp \alpha\cdot s_l}f_\fp(\alpha)d s_\alpha \mathbb{P}_l^\top\hat{x}\label{limp} \\
	&+\frac{4\mu\omega^2}{\lambda+2\mu}\int_\ms\me^{\mi\kp \alpha\cdot s_l}f_\fp(\alpha)\alpha\otimes\alpha d s_\alpha \mathbb{P}_l^\top \hat{x}+\omega^2\int_\ms\me^{\mi\ks \alpha\cdot s_l}f_\fs(\alpha)\alpha^\bot
	d s_\alpha |D_l|\notag\\
	&+2\kp\ks\mu\int_\ms\me^{\mi\ks \alpha\cdot s_l}f_\fs(\alpha)\alpha^\bot\otimes\alpha d s_\alpha \mathbb{P}_l^\top \hat{x}+2\kp\ks\mu\int_\ms\me^{\mi\ks \alpha\cdot s_l}\alpha\otimes f_\fs(\alpha)\alpha^\bot d s_\alpha \mathbb{P}_l^\top \hat{x}\bigg],\notag
\end{align}
\begin{align}
	g_{l,\fs}(\hat{x}) &=\ks^{\frac{3}{2}}\me^{-\mi\ks\hat{x}\cdot s_l}\hat{x}^\bot\cdot\bigg[\omega^2\int_{\mathbb{S}}\me^{\mi\kp \alpha\cdot s_l}f_\fp(\alpha)\alpha
	d s_\alpha |D_l|+{2\lambda\kp\ks}\int_{\mathbb{S}}\me^{\mi\kp \alpha\cdot s_l}f_\fp(\alpha)d s_\alpha \mathbb{P}_l^\top \hat{x}\label{lims}\\
	&+{4\mu\kp\ks}\int_{\mathbb{S}}\me^{\mi\kp \alpha\cdot s_l}f_\fp(\alpha)\alpha\otimes\alpha d s_\alpha \mathbb{P}_l^\top \hat{x}+\omega^2\int_{\mathbb{S}}\me^{\mi\ks \alpha\cdot s_l}f_\fs(\alpha)\alpha^\bot
	d s_\alpha |D_l|\notag\\
	&+2\omega^2\int_{\mathbb{S}}\me^{\mi\ks \alpha\cdot s_l}f_\fs(\alpha)\alpha^\bot\otimes\alpha d s_\alpha \mathbb{P}_l^\top \hat{x}+2\omega^2\int_{\mathbb{S}}\me^{\mi\ks \alpha\cdot s_l}\alpha\otimes f_\fs(\alpha)\alpha^\bot d s_\alpha \mathbb{P}_l^\top \hat{x}\bigg].\notag
\end{align}
It is easy to verify that
\[||-\gamma^{-1}\rho^{-2}F^\rho-F^0||_{\ml}=\sup_{f\in \ml,||f||^2_2=1}||(-\gamma^{-1}\rho^{-2}F^\rho-F^0)f||^2_{\ml}=\oo(\rho\ln\rho+L^{-1/2}).\]
Since $F^\rho$ is compact and normal, perturbation theory~\cite{kato2013perturbation} ascertains the continuity of any finite system of eigenvalues as well as of the associated eigen-projection. In other words, we can study on the eigensystems of the limit far operator $F^0$ instead of the original $F^\rho$~\cite{hazardSelectiveAcousticFocusing2004}.

Before finding the eigensystem of $F^0$, we first recall a classical result of oscillatory integrals that will be very useful for our analysis. It can be seen
in \cite{steinHarmonicAnalysisRealVariable1993} and restated in \cite{burkardFarFieldModel2013} as follows.
\begin{theorem}[Burkard,Minut,Ramdani]\label{BMR}
	Let $S$ be a smooth hypersurface in $\mathbb{R}^N$ whose Gaussian curvature  is nonzero everywhere and let $\psi\in C_0^\infty(\mathbb{R}^N)$ such that
	$Supp(\psi)$ intersects $S$ in a compact subset of $S$. Then, as
	$|\xi|\rightarrow+\infty$, we have
	\begin{align*}
		\int_S\psi(\alpha)\me^{\mi\alpha\cdot\xi}d\alpha=\oo(|\xi|^{\frac{1-N}{2}}).
	\end{align*}
\end{theorem}
%\vspace{1em}

Now, we are at the position to describe the eigensystem of $F^0$.
%\begin{theorem}
%	For $l\in\{1,\cdots,M\}$, let 
%\end{theorem}
From the explicit expression \eqref{limp} and \eqref{lims} of $F^0$, we can see that
\begin{align}
	\begin{aligned}
		\mathcal{R}(F^0)\subset\mathop{\mathrm{\bigoplus}}\limits_{1\leq l\leq M}\left(\mathcal{A}_l\oplus \mathcal{B}_l\right),
	\end{aligned}
\end{align}
where 
\begin{align}
	\begin{aligned}
		\mathcal{A}_l=&\textrm{span}\{(\kp^{\frac{3}{2}}\hat{x}\cdot\me^{-\mi\kp\hat{x}\cdot s_l}\overrightarrow{c},\ks^{\frac{3}{2}}\hat{x}^\bot\cdot\me^{-\mi\ks\hat{x}\cdot s_l}\overrightarrow{c})|\overrightarrow{c}\in\mathbb{C}^2\},\\
		\mathcal{B}_l=&\textrm{span}\{
		(\kp^{\frac{5}{2}}\hat{x}\cdot A\hat{x}\me^{-\mi\kp\hat{x}\cdot s_l},\ks^{\frac{5}{2}}\hat{x}^\bot\cdot A\hat{x}\me^{-\mi\ks\hat{x}\cdot s_l})|A\in\mathbb{C}^{2\times 2}\},
	\end{aligned}
\end{align}
which indicates the range dimension of $F^0$ is at most $6M$. A direct conclusion in linear
algebra shows that there exists at most $6M$ non-zero eigenvalues for the limit far field operator $F^0$. In the following, we search for eigenfunctions lie in $\mathcal{A}_l$ and $\mathcal{B}_l$ with $1\leq l\leq M$, respectively.

\noindent\textbf{Case \Rmnum{1}:} Consider the eigenvectors in the form of $g_l=(\kp^{\frac{3}{2}}\hat{x}\cdot\me^{-\mi\kp\hat{x}\cdot s_l}\overrightarrow{c},\ks^{\frac{3}{2}}\hat{x}^\bot\cdot\me^{-\mi\ks\hat{x}\cdot s_l}\overrightarrow{c})$ in $\mathcal{A}_l$.
After a straightforward insertion with the help of Theorem \ref{BMR}, we obtain
\begin{align}
	F^0(g_l)=\pi\omega^2|D_l|(\kp^{\frac{3}{2}}+\ks^{\frac{3}{2}})(\kp^{\frac{3}{2}}\hat{x}\cdot\me^{-\mi\kp\hat{x}\cdot s_l}\overrightarrow{c},\ks^{\frac{3}{2}}\hat{x}^\bot\cdot\me^{-\mi\ks\hat{x}\cdot s_l}\overrightarrow{c})+\oo((kL)^{-\frac{1}{2}}).
\end{align}
Therefore, we get two approximately equal eigenvalues $\lambda_{l,1}=\lambda_{l,2}=\pi\omega^2|D_l|(\kp^{\frac{3}{2}}+\ks^{\frac{3}{2}})$,
of which the corresponding eigenspace is $\textrm{span}\{(\kp^{\frac{3}{2}}\hat{x}\cdot\me^{-\mi\kp\hat{x}\cdot s_l}\overrightarrow{c},\ks^{\frac{3}{2}}\hat{x}^\bot\cdot\me^{-\mi\ks\hat{x}\cdot s_l}\overrightarrow{c})|\overrightarrow{c}\in\mathbb{C}^2\}$. For simplicity, we can select two
linearly independent elements \begin{align}\label{glp}
	{g}_{l,p}=(\kp^{\frac{3}{2}}\hat{x}\cdot\me^{-\mi\kp\hat{x}\cdot s_l}{e}_p,\ks^{\frac{3}{2}}\hat{x}^\bot\cdot\me^{-\mi\ks\hat{x}\cdot s_l}{e}_p),\quad p=1,2,
\end{align} as its eigenvectors.

\noindent\textbf{Case \Rmnum{2}:} Find the eigenvectors in the space $\mathcal{B}_l$, namely, in the form of 
\begin{align}\label{hl}
	h_l=(\kp^{\frac{5}{2}}\hat{x}\cdot A\hat{x}\me^{-\mi\kp\hat{x}\cdot s_l},\ks^{\frac{5}{2}}\hat{x}^\bot\cdot A\hat{x}\me^{-\mi\ks\hat{x}\cdot s_l}),
\end{align} 
which is slightly more complicated than Case \Rmnum{1}.

By the insertion of \eqref{hl} into the definition of $F^0$\eqref{limoper}, we obtain that,
\begin{align*}
	(F^0h_l)_\fp(\hat{x})&=\kp^{\frac{5}{2}}\me^{-\mi\kp\hat{x}\cdot s_l}\hat{x}\cdot\bigg[2\lambda\kp^{\frac{7}{2}}\int_{\mathbb{S}}\alpha \cdot A\alpha d s_\alpha \mathbb{P}_l^\top \hat{x}+4\mu\kp^{\frac{7}{2}}\int_{\mathbb{S}}(\alpha\cdot A\alpha)\alpha\otimes\alpha d s_\alpha \mathbb{P}_l^\top \hat{x}\\
	&+2\mu\ks^{\frac{7}{2}}\int_{\mathbb{S}}(\alpha^\bot\cdot A \alpha)(\alpha^\bot\otimes \alpha+\alpha\otimes \alpha^\bot) d s_\alpha \mathbb{P}_l^\top \hat{x}\bigg]+\oo((kL)^{-\frac{1}{2}}),\\
	(F^0h_l)_\fs(\hat{x})&=\ks^{\frac{5}{2}}\me^{-\mi\ks\hat{x}\cdot s_l}\hat{x}^\bot\cdot\bigg[2\lambda\kp^{\frac{7}{2}}\int_{\mathbb{S}}\alpha \cdot A\alpha d s_\alpha \mathbb{P}_l^\top \hat{x}+4\mu\kp^{\frac{7}{2}}\int_{\mathbb{S}}(\alpha\cdot A\alpha)\alpha\otimes\alpha d s_\alpha \mathbb{P}_l^\top \hat{x}\\
	&+2\mu\ks^{\frac{7}{2}}\int_{\mathbb{S}}(\alpha^\bot\cdot A \alpha)(\alpha^\bot\otimes \alpha+\alpha\otimes \alpha^\bot) d s_\alpha \mathbb{P}_l^\top \hat{x}\bigg]+\oo((kL)^{-\frac{1}{2}}).
\end{align*}
We select the basis of $A\in \mathbb{C}^{2\times2}$ as follows
\[A_1=\begin{pmatrix}
	1&0\\0&0
\end{pmatrix},A_2=\begin{pmatrix}
	0&0\\0&1
\end{pmatrix},A_3=\begin{pmatrix}
	0&1\\1&0
\end{pmatrix},A_4=\begin{pmatrix}
	0&1\\-1&0
\end{pmatrix}.\]
A straightforward calculation shows that $F^0h_{l,4}=0$. Thus there are at most 3 non-zero approximate eigenvectors in $\mathcal{B}_l$.

In order to find eigenvectors corresponding to non-zero eigenvalues in $\mathcal{B}_l$, it is now necessary to consider the  restriction operator of $F^0$
from $\ml$ to $\textrm{span}\{h_{l,i},i=1,2,3,4\}$ which is an invariant subspace for $F^0$.
After a tedious algebra, as shown in appendix \ref{appB}, we get the matrix representation,
%We are at the position to obtain the matrix representation of the liner operator.
\begin{align}
	F^0(h_{l,1},h_{l,2},h_{l,3},h_{l,4})=(h_{l,1},h_{l,2},h_{l,3},h_{l,4})\mathbb{F},
\end{align}
where $\mathbb{F}$ is given by \eqref{matf}. We see that for each $l=1,\cdots,M$, $\mathbb{F}$ has a zero eigenvalue with the eigenvector $(0,0,0,1)^\top$. What we want to obtain is the rest so-called “significant” eigenvalues and
corresponding eigenvectors. For simplicity, we denote the other three eigenvalues by $\zeta_{l,q},q=1,2,3$ and the corresponding eigenvectors by $v_{l,q},q=1,2,3$ respectively, i.e.
\[\mathbb{F}v_{l,q}=\zeta_{l,q}v_{l,q},\quad q=1,2,3.\]
Thus, we get the $3$ eigenvectors in terms of $\mathcal{B}_l$ as follows
\begin{align}\label{eigb}
	h_{l,q}=\left(\kp^{\frac{5}{2}}\hat{x}\cdot A_{l,q}\hat{x}\me^{-\mi\kp\hat{x}\cdot s_l},\ks^{\frac{5}{2}}\hat{x}^\bot\cdot A_{l,q}\hat{x}\me^{-\mi\ks\hat{x}\cdot s_l}\right),\quad q=1,2,3,
\end{align}
where
\[A_{l,q}=\sum_{j=1}^4 v_{l,q}^jA_j=\begin{pmatrix}
	v_{l,q}^1&v_{l,q}^3+v_{l,q}^4\\
	v_{l,q}^3-v_{l,q}^4&v_{l,q}^2
\end{pmatrix},\]
and $v_{l,q}^j$ denote the $j$-th element of $v_{l,q}$. 

Combining case \Rmnum{1} and \Rmnum{2}, we give the following conclusion.
\begin{theorem}\label{feig}
	For all $l=1,\cdots,M$ and all $p=1,2,q=1,2,3$, we recall the functions 
	${g}_{l,p}\in\ml$ defined in \eqref{glp} and $h_{l,q}\in\ml$ defined in \eqref{eigb}.
	%	\begin{align}\label{eigv}
		%		e_l(\alpha)=\me^{-\mi k\alpha\cdot s_l},\qquad g_{l,q}= h_{l,q}(\alpha)e_p(\alpha),
		%	\end{align}
	%	where $h_{l,q}\in L^2(\ms)$ is introduced in statement(\rmnum{4}) in Lemma \ref{eim}.
	We also set
	\begin{align*}
		\lambda_{l,p}^a=\pi\omega^2|D_l|(\kp^{\frac{3}{2}}+\ks^{\frac{3}{2}}),\qquad \lambda_{l,q}^b=\zeta_{l,q}.
	\end{align*}
	Then, as $kL\rightarrow\infty$, we have the following two results.
	\begin{enumerate}
		\item The functions ${g}_{l,p}$ satisfy
		\begin{align}\label{eig1}
			F^0{g}_{l,p}=\lambda_{l,p}^a{g}_{l,p}+\oo((kL)^{-\frac{1}{2}}),\qquad p=1,2.
		\end{align}
		\item The functions $h_{l,q}$ satisfy
		\begin{align}\label{eig2}
			F^0h_{l,q}=\lambda_{l,q}^bh_{l,q}+\oo((kL)^{-\frac{1}{2}}),\qquad q=1,2,3.
		\end{align}
	\end{enumerate}
\end{theorem}

Summing up, we have proved that in the regime of small and well-resolved elastic cavities, the limit far field operator $F^0$ admits $5M$ approximate non-zero eigenvalues:
$\lambda_{l,p}^a,\lambda_{l,q}^b$ for $l=1,\cdots,M,p=1,2,q=1,2,3$.
\begin{remark}
	In the case when cavities are disks, $\{{g}_{l,p},h_{l,q}\}$ can be explicitly constructed since the polarizability tensor $\mm_l$ is already diagonal. For the general case, $\{{g}_{l,p},h_{l,q}\}$ must be found through numerical computation. Thus, Theorem \ref{feig} can be taken as a generalization of Theorem \ref{smalleig} to general shaped cases.
\end{remark}

By combining the results of Theorems \ref{feig} and \ref{BMR}, the following theorem  provides the expected selective focusing properties with the eigenfunctions of the far field operator(thus of time reversal operator).

\begin{theorem}\label{gloo}
	For $1\leq l\leq M$, the incident Herglotz wave associated with the approximate eigenfunctions ${g}_{l,p},p=1,2$ and $h_{l,q},q=1,2,3$ given by \eqref{glp} and \eqref{eigb}
	will selectively focus on the $l$th cavity.
\end{theorem}

The proof is a simple application of Theorem \ref{BMR} since the incident waves
$$	\boldsymbol{u}^{\boldsymbol{i}}_{g_{l,p}}(x) =\int_{\ms} \kp^{\frac{3}{2}}\alpha\otimes\alpha{\me_p}\me^{\mi \kp \alpha\cdot (x-s_l)}  +\ks^{\frac{3}{2}}\alpha^\bot\otimes\alpha^\bot{\me_p}\me^{\mi \ks\alpha \cdot (x-s_l)} ds_{\alpha} $$
and $$\boldsymbol{u}^{\boldsymbol{i}}_{h_{l,q}}(x)=\int_{\ms}\kp^{\frac{5}{2}} (\alpha\cdot A_{l,q}\alpha)\alpha\me^{\mi \kp \alpha\cdot (x-s_l)}  +\ks^{\frac{5}{2}}(\alpha^\bot\cdot A_{l,q}\alpha)\alpha^\bot\me^{\mi \ks\alpha \cdot (x-s_l)} ds_{\alpha}$$ decay rapidly away from the $l$th particle as $\oo((k|x-s_l|)^{-\frac{1}{2}})$.
\begin{remark}
	Clearly, similar selective focusing results hold for the there dimensional problem. We actually find that each small cavity gives rise to nine significant eigenvalues in three dimensions. Details will be given in a forthcoming paper.
\end{remark}

\section{Numerical experiments}\label{exper}
In this section, we test our theoretical results and the efficiency of DORT method in several examples. The general procedure is: first apply the forward problem solver to obtain the far field data, and then use the DORT method to recover the locations and shapes of cavities. We use the so-called
spectral approach (also known as multipole expansion or Mie series method)\cite{thierryDiffOpensourceMatlab2015a} for the discal case, otherwise we compute by boundary element method. For the inversion part, to avoid inverse crime and test the robustness of the algorithm, we add 5\% Gaussian noise to the simulated far field data. Once the numerical far field operator $F$ is found, the eigenvalues and eigenvectors are obtained through
the ‘$eig$’ command in MATLAB, and Herglotz wave $\uincf$ is evaluated in a straightforward manner, although NUFFT\cite{greengardAcceleratingNonuniformFast2004}
can be used to accelerate the evaluation if the time to evaluate $\uincf$ becomes expensive.

Throughout all the examples, the Lam\'e constants and angular frequency are chosen to be $\lambda=1,\mu=2,\omega=2$, leading to the compressional and shear wavenumbers $\kp=\frac{2\sqrt{5}}{5},\ks=\sqrt{2}$, respectively. The emission and reception directions are obtained by discretizing uniformly the unit circle $[0,2\pi]$ using 360 points. 
\subsection{Example 1: a single cavity}
We first consider the case of one single discal cavity located at $(5,0)$ with radius $R=0.002$. Then we consider a peanuthull shape with boundary given by 
\[\begin{cases*}
	x=5+0.002(2+\sin(2t))\cos(t),\\
	y=0.002(2+\sin(2t))\sin(t),
\end{cases*}\qquad t\in [0,2\pi].\]
We show in table \ref{lamdis} the first six largest eigenvalues of the time reversal operator for these two different shapes. In particular, we observe that for each setup, they admits five significant eigenvalues respectively, which is in agreement with our theoretical result. Figure \ref{ex1} shows the modulus of first thirty eigenvalues and the magnitude of the Herglotz waves with eigenfunctions corresponding to the first five eigenvalues as kernels for a single disk, so does figure \ref{ex2} for a single peanuthull. We see that the Herglotz wave function associated to these eigenvalues achieves its maximum at the location of the unknown cavity. On the other hand, the shape reconstruction is relatively difficult due to the small size of the cavity.

In order to observe the phenomenon when the cavity is relatively large, we set the boundary of the obstacle as follows
\[\begin{cases*}
	x=3(2+\sin(2t))\cos(t),\\
	y=3(2+\sin(2t))\sin(t),
\end{cases*}\qquad t\in [0,2\pi].\]
From figure \ref{ex3}, we can see that there are not five significant eigenvalues anymore, but the global focusing is still achieved. We also observe that compared to the small cavity, it is easier to reconstruct the shape in this case.
\begin{table}[ht]
	\footnotesize
	\caption{Significant eigenvalues of the time reversal operator for a disk and a peanuthull.}
	\label{lamdis}
	\centering
	\begin{tabular}{|c|c|c|c|c|c|c|}\hline
		Shape&$\lambda_1$&$\lambda_2$&$\lambda_3$&$\lambda_4$&$\lambda_5$&$\lambda_6$\\ \hline
		Disk&$0.005021$&$0.005021$&$0.000670$&$0.000670$&$0.000472$&$4.1688\me-15$\\ \hline
		Peanuthull&$0.002011$&$0.001929$&$0.000167$&$0.000167$&$0.000111$&$6.3113\me-16$\\ \hline
	\end{tabular}
	%	\vspace{1em}
\end{table}
%%\begin{figure}
%%	\centering
%%\includegraphics[scale=0.7]{test.jpg}
%%\end{figure}
\begin{figure}
	\centering
	\includegraphics[scale=0.8]{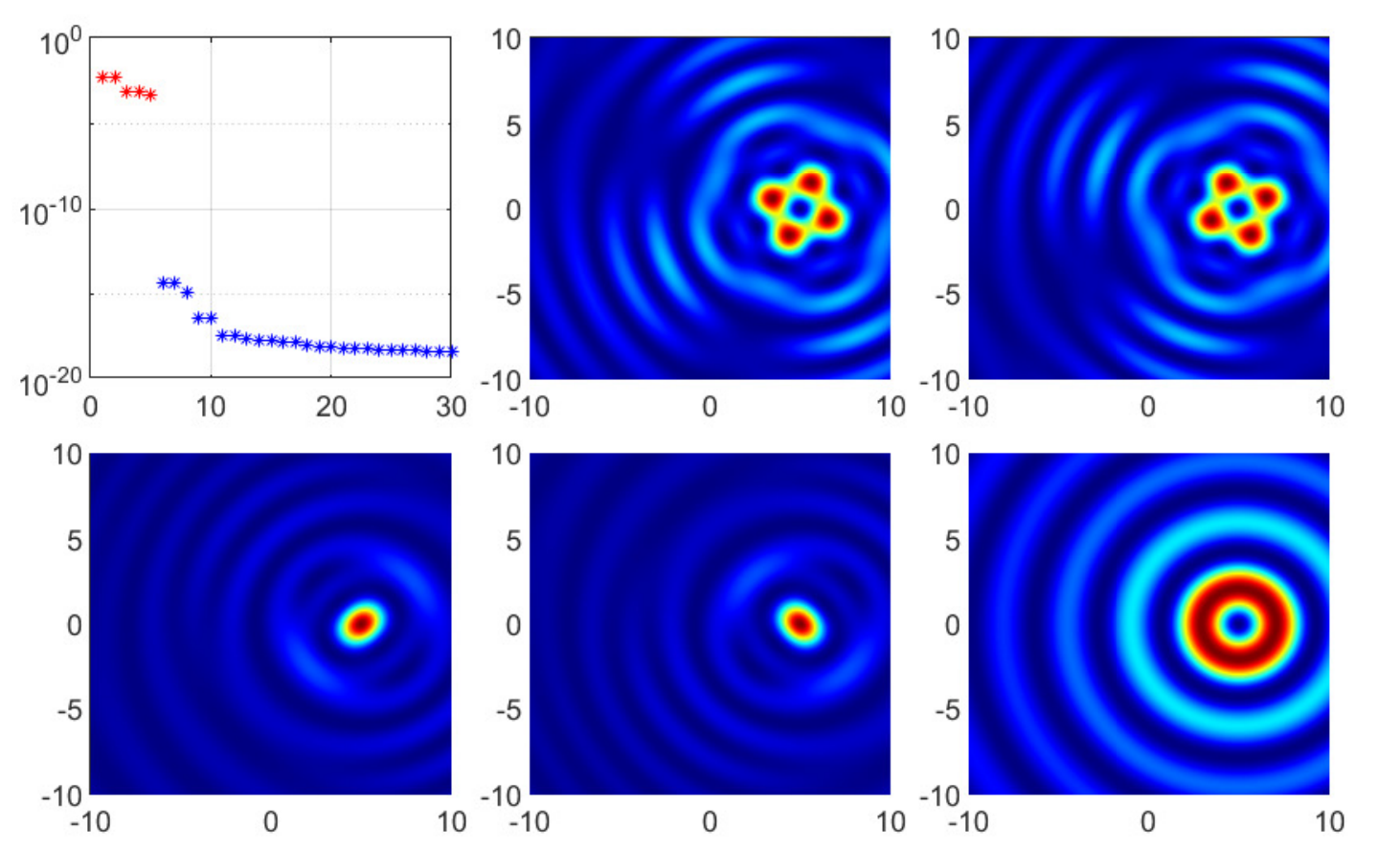}
	\caption{Imaging of a single disk: Eigenvalues of the time reversal operator(top left) and the Herglotz waves with eigenfunctions of the first five eigenvalues (the others).}\label{ex1}
\end{figure}
\begin{figure}
	\centering
	\includegraphics[scale=0.8]{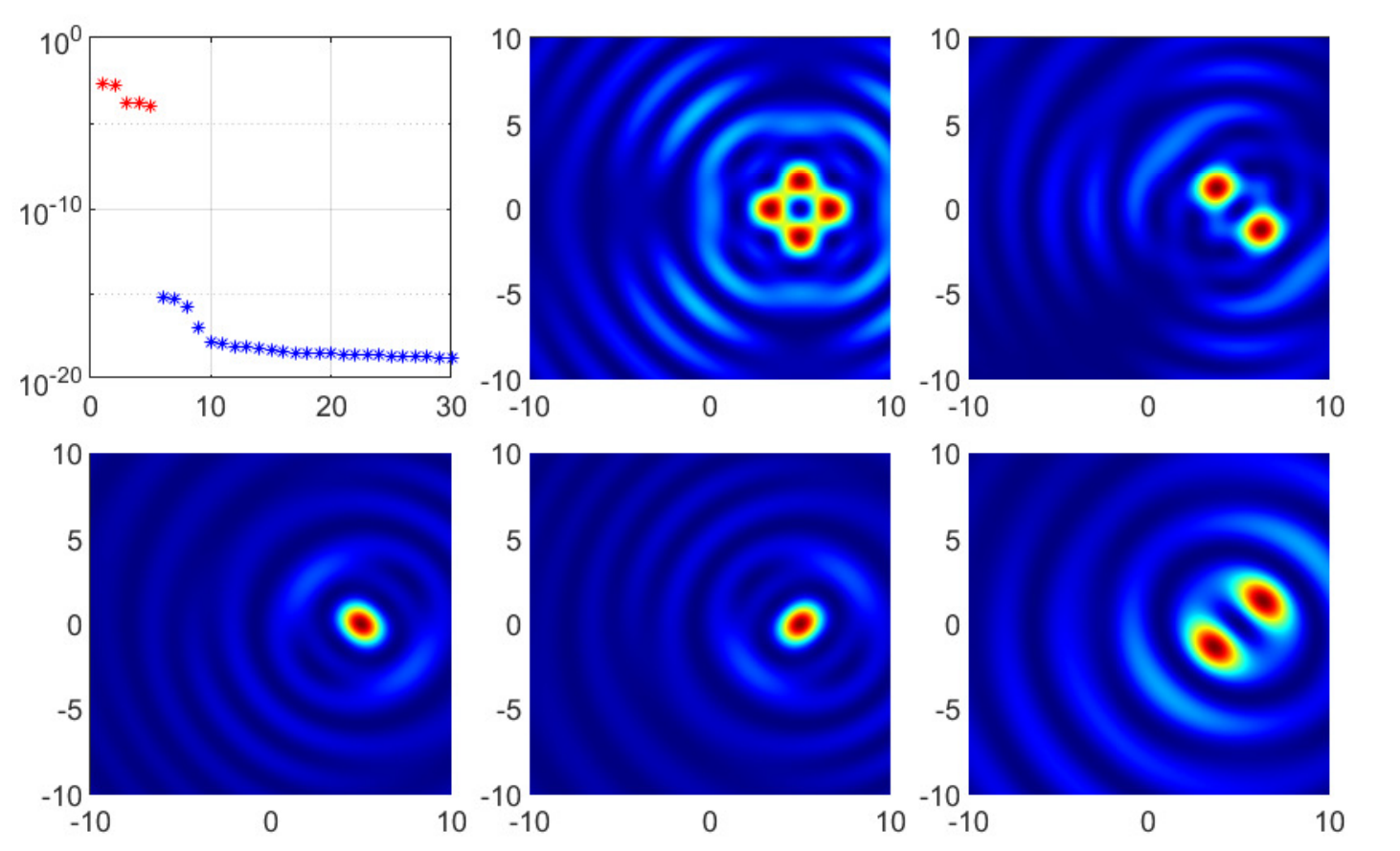}
	\caption{Imaging of a single peanuthull: Eigenvalues of the time reversal operator (top left) and the Herglotz waves with eigenfunctions of the first five eigenvalues (the others).}\label{ex2}
\end{figure}
\begin{figure}
	\centering
	\includegraphics[scale=0.7]{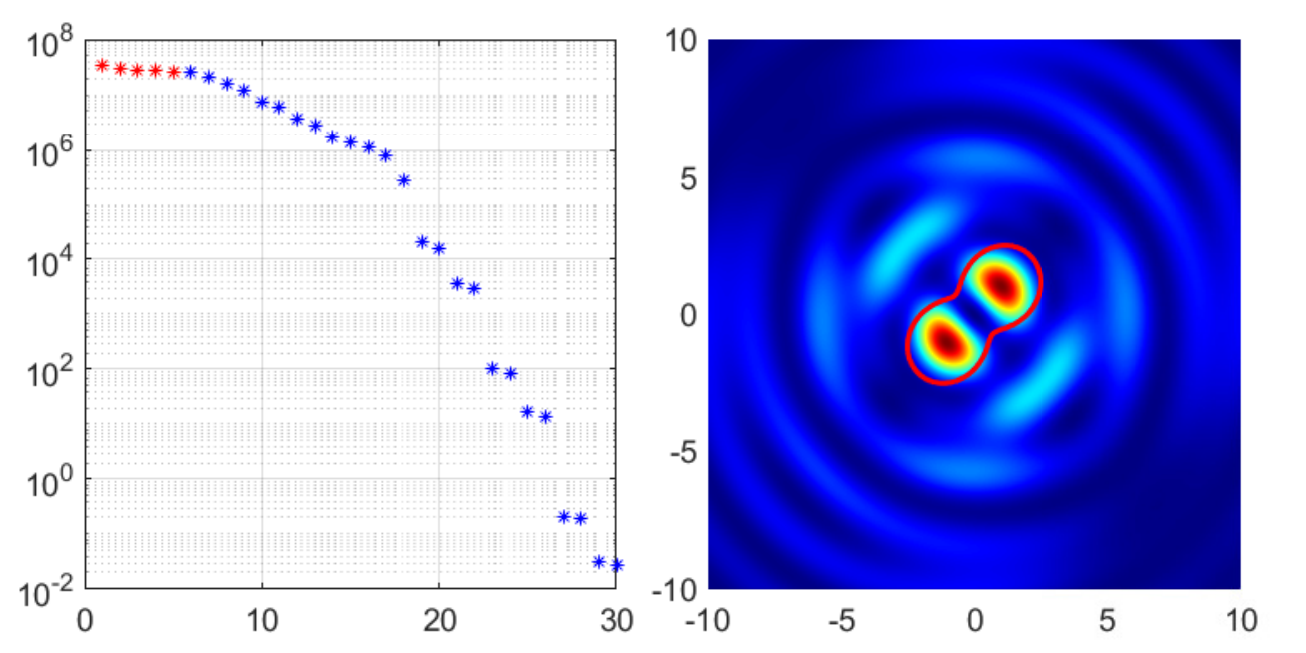}
	\caption{Imaging of a single large peanuthull: Eigenvalues of the time reversal operator (left) and the Herglotz waves with eigenfunction of the first eigenvalue (right). The red curve on the right is the exact shape.}\label{ex3}
\end{figure}
\subsection{Example 2: a pair of discal cavities}
Let the first discal cavity be given as the same as example 1 and consider the second cavity located at the position $(-5,0)$ with radius $R=0.004$.
In figure \ref{ex4} one can see that the first ten eigenvalues (ordered by their magnitude) of the time reversal mirror are significantly larger than the subsequent eigenvalues. Moreover, the first five eigenvalues are due to the first cavity, as well as the next five eigenvalues due to the second cavity. This indicates that the Herglotz waves associated to the first five eigenvalues all focus on one cavity and the Herglotz waves generated from eigenfunctions corresponding to the next five eigenvalues should focus on the other cavity. Indeed, as can be seen in figure \ref{ex5}, this is the case that one expects from the theoretical results.
\begin{figure}
	\centering
	\includegraphics[scale=0.7]{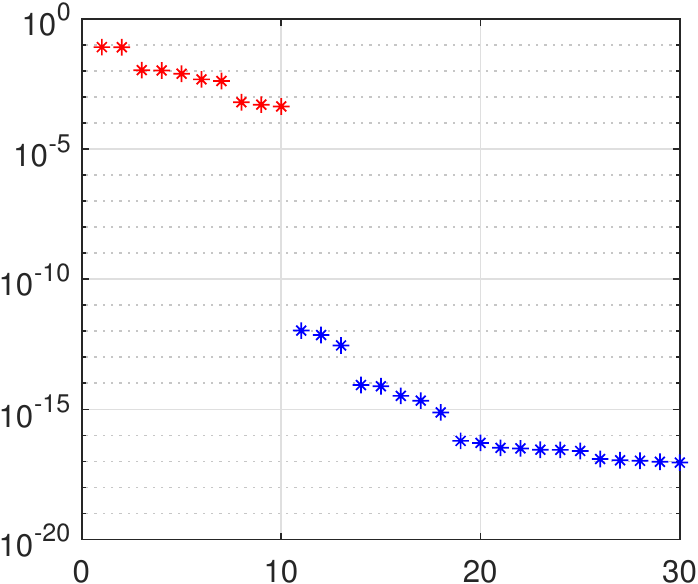}
	\caption{Eigenvalues of the time reversal operator in the case of two disks with different sizes.}\label{ex4}
\end{figure}
\begin{figure}
	\centering
	\includegraphics[scale=1]{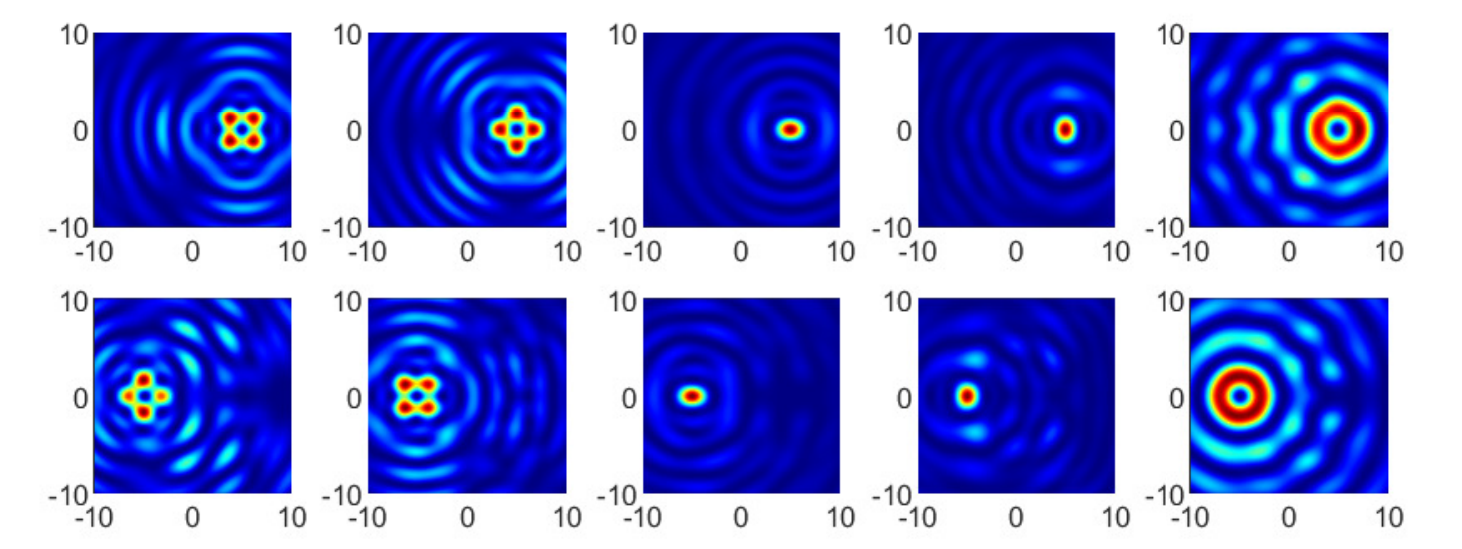}
	\caption{Imaging of two disks with different sizes: Herglotz waves with eigenfunctions of the first five eigenvalues (top row) and the next five eigenvalues (bottom row).}\label{ex5}
\end{figure}
On the other hand, if the two elastic cavities $D_1$ and $D_2$ located at $(5,0)$ and $(-5,0)$ with the same radius $R=0.002$, the selective focusing will disappear due to symmetry. The obtained eigenvalues are given in figure \ref{ex6}, which still shows ten significant eigenvalue. The associated Herglotz waves are presented in figure \ref{ex7}. Due to the fact that the eigenvalues associated to $D_1$ are equal to those associated to $D_2$, we observe that the Herglotz waves do not selectively focus on one but on both of the cavities.
\begin{figure}
	\centering
	\includegraphics[scale=0.7]{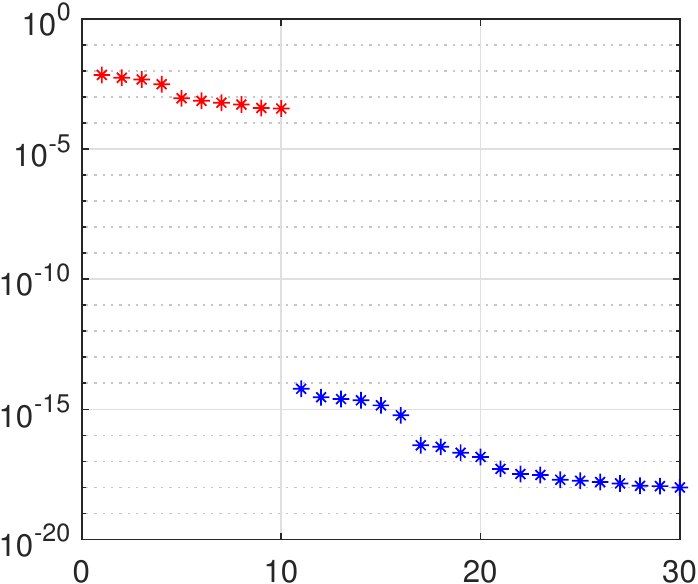}
	\caption{Eigenvalues of the time reversal operator in the case of two disks with the same size.}\label{ex6}
\end{figure}
\begin{figure}
	\centering
	\includegraphics[scale=1]{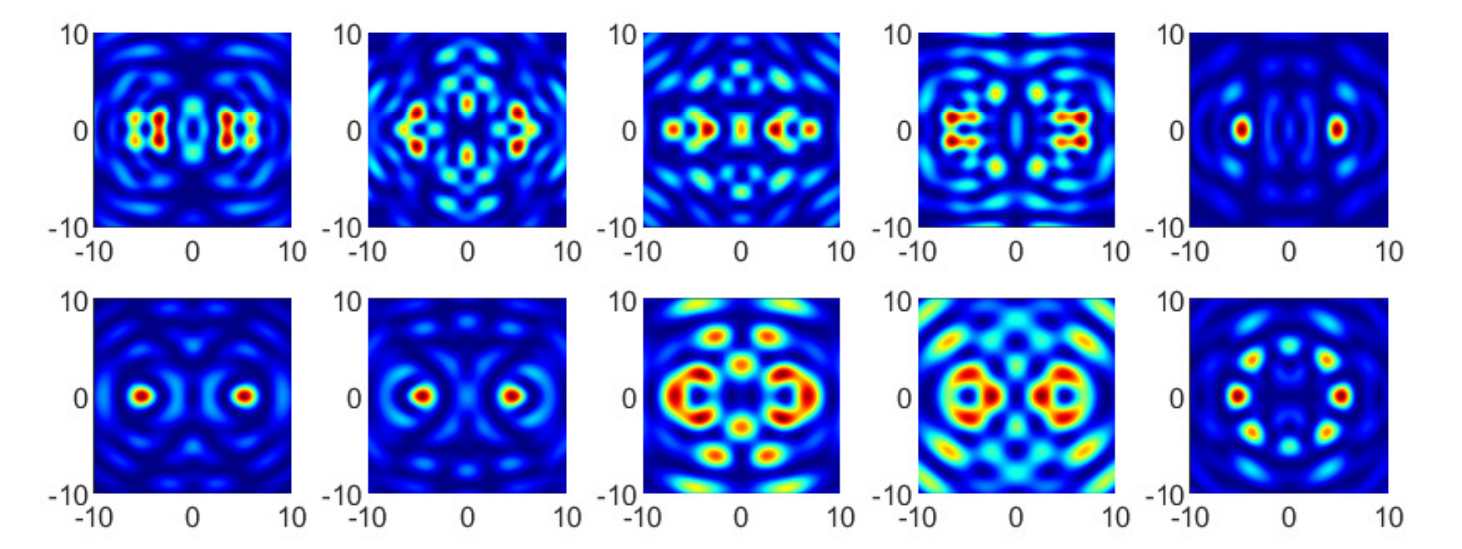}
	\caption{Imaging of two disks with the same size: Herglotz waves with eigenfunctions of the first ten eigenvalues.}\label{ex7}
\end{figure}

\subsection{Example 3: nine cavities with different sizes}
To further show the efficiency of DORT method, we give a numerical example of recovering the locations of nine disks with different sizes given in table \ref{tab2}. We present the eigenvalues in figure \ref{ex12}, which shows that the first forty-five dominant eigenvalues (red points) are in agreement with our theoretical results. In figure \ref{ex13}, the absolute values of the associated Herglotz waves are shown, which indicate that the DORT method recovers the location of the unknown disks accurately.
\begin{table}[h]
	\footnotesize
	\caption{Distribution of the nine disks.}
	\label{tab2}
	\centering
	\begin{tabular}{|c|c|c|c|c|c|c|c|c|c|}\hline
		Radius&$0.01$&$0.02$&$0.03$&$0.04$&$0.05$&$0.06$&$0.07$&$0.08$&$0.09$\\ \hline
		Center&$(-12,12)$&$(0,12)$&$(12,12)$&$(-12,0)$&$(0,0)$&$(12,0)$&$(-12,-12)$&$(0,-12)$&$(12,-12)$\\ \hline
	\end{tabular}
\end{table}
\begin{figure}
	\centering
	\includegraphics[scale=0.7]{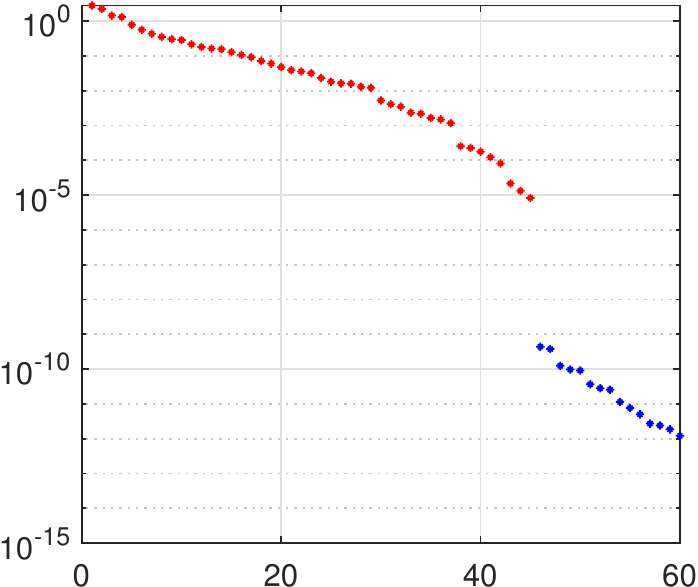}
	\caption{Eigenvalues of the time reversal operator in the presence of nine disks.}\label{ex12}
\end{figure}
\begin{figure}
	\centering
	\includegraphics[scale=0.7]{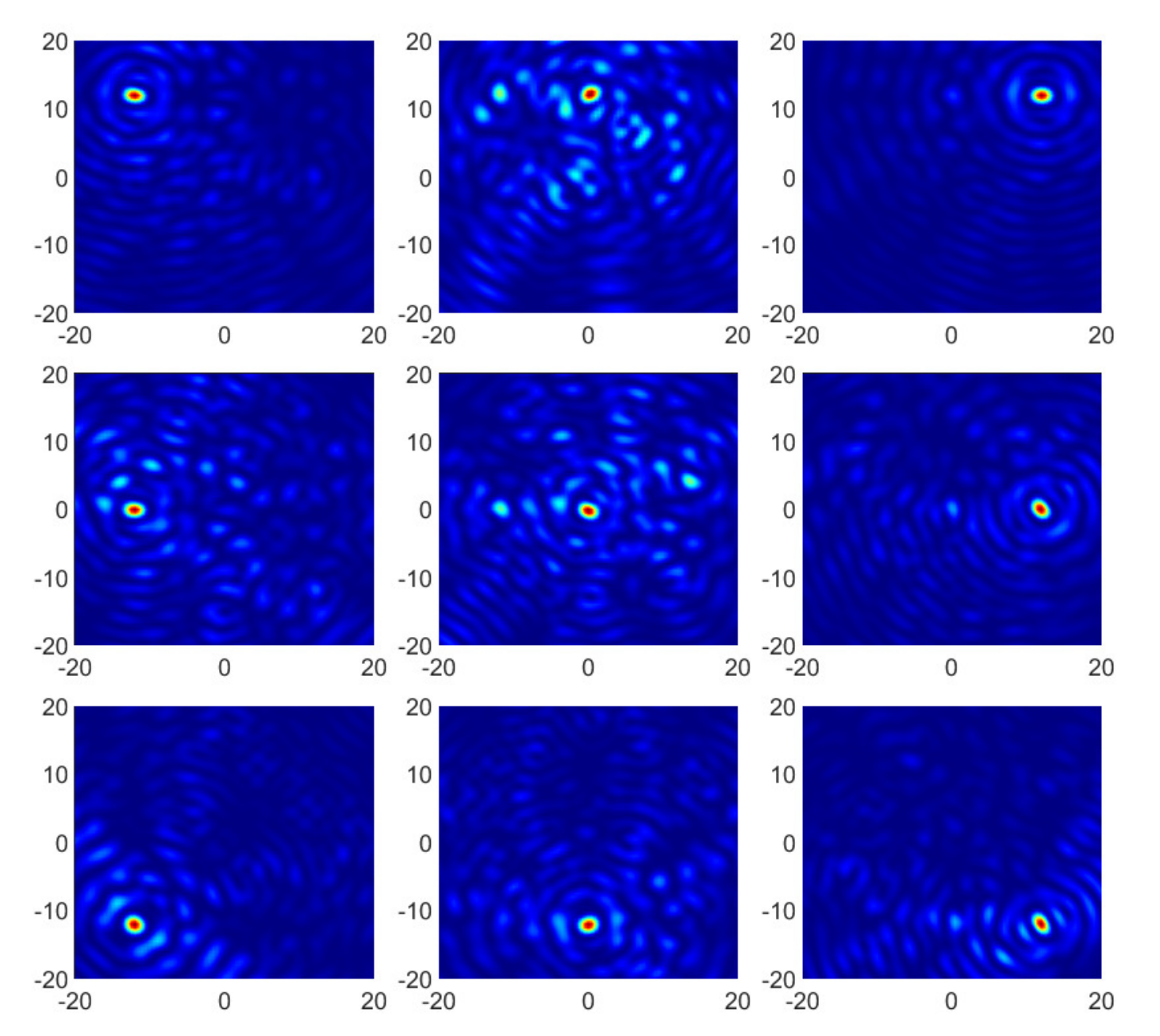}
	\caption{Imaging of nine disks with different sizes: Herglotz waves with eigenfunctions corresponding to nine dominant eigenvalues selected from the first forty-five eigenvalues.}\label{ex13}
\end{figure}

\subsection{Example 4: two asymmetric cavities with a limited aperture data}
In the last example, we test the DORT method with an open time reversal mirror. In other words, we only use limited aperture far field data. Assume the two different sized discal cavities are given as the same as in Example 2 and an open TRM  is given with the measured aperture
\begin{align}\label{opensym}
	\ps=\{\alpha\in\ms|\alpha=(\cos\phi,\sin\phi),\phi\in[\pi/4,3\pi/4]\cup[5\pi/4,7\pi/4]\}.
\end{align}
In figure \ref{ex8} we present the eigenvalues obtained from the open TRM \eqref{opensym}. In figure \ref{ex9}, the absolute values of the associated Herglotz waves are shown. One can see that the selective focusing is still achieved, however, with a lot of oscillations at the aperture that has no measured data.
% Rigorous mathematical analysis for such phenomenon will be explored in the future.
\begin{figure}
	\centering
	\includegraphics[scale=0.7]{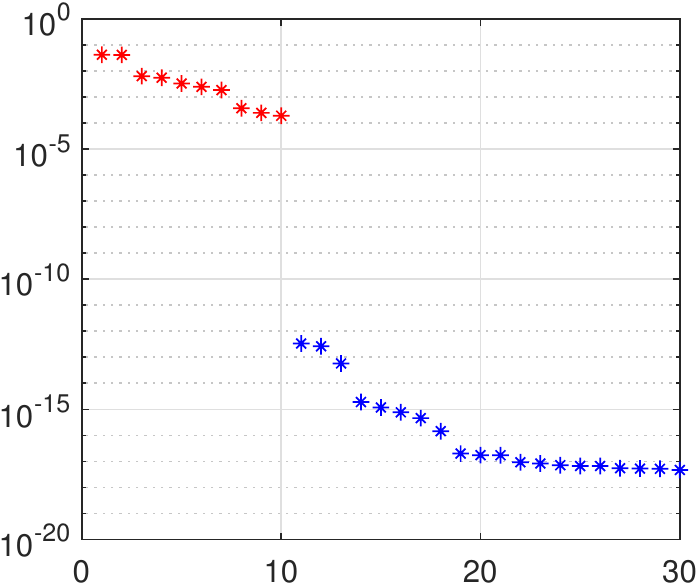}
	\caption{Eigenvalues of the time reversal operator in the case of an open TRM.}\label{ex8}
\end{figure}
\begin{figure}
	\centering
	\includegraphics[scale=1]{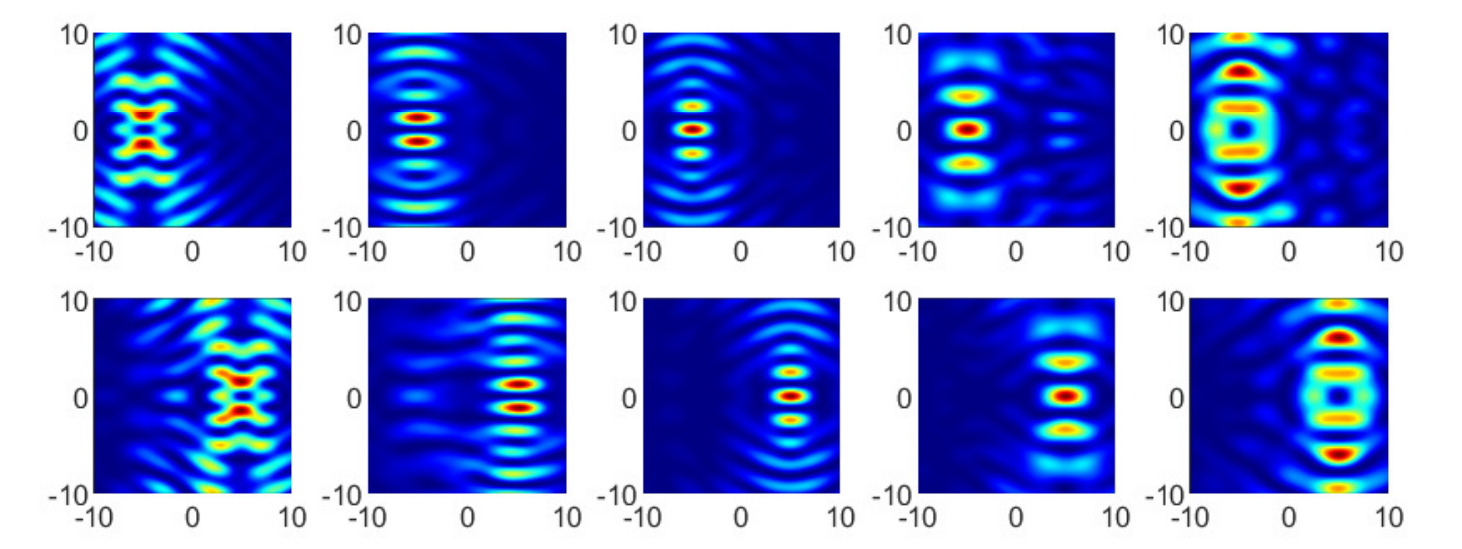}
	\caption{Imaging of two disks in the case of an open TRM: Herglotz waves with eigenfunctions of the first ten eigenvalues.}\label{ex9}
\end{figure}

\section{Conclusion}
In this work, we presented a rigorous mathematical justification for the DORT method in the context of small elastic cavities in two dimensions based time harmonic elastic scattering model. The main result states that each cavity gives rise to 5 significant eigenvalues whose corresponding eigenvectors can be used to achieve selective focusing, provided the cavities are small and well-resolved. These theoretical results were further confirmed by numerical simulations. One can see that the DORT method can effectively recover the locations of the unknown cavities with a slightly less satisfactory result on the shape reconstruction, especially in the case of limited-aperture data. Therefore, our future work includes the analysis for selective focusing with limited far field data as well as the approach to improve the shape recovery when imaging multiple cavities.

\appendix

\section{Scattering matrix of a discal cavity}\label{appA}
The scattering matrix $\mathcal{S}$ for a disk  $S_0$ is given in the form of
\begin{eqnarray}\label{scat_mat} \mathcal{S}=\mbox{diag}[\cdots,\mathcal{S}_{-1},\mathcal{S}_{0},\mathcal{S}_{1},\cdots],
\end{eqnarray}
where each $\mathcal{S}_{n}$ is a $2\times 2$ matrix block.
According to the boundary condition \eqref{rigidboundary} on $S_0$,  mode matching yields 
\begin{eqnarray}\label{scat_spher}
	\begin{bmatrix}
		\alpha_{n} \\
		\beta_{n} \\
	\end{bmatrix}
	=\mathcal{S}_{n}
	\begin{bmatrix}
		a_{n} \\
		b_{n} \\
	\end{bmatrix}, n \in\mathbb{Z},
\end{eqnarray} 
%	with $\alpha_{0,0}=\beta_{0,0}=a_{0,0} = b_{0,0}=0$. It holds from \cite{Louer2014} that $\gamma_{0,0} = -\frac{j'_0(\kp R)}{{h_0^{(1)}}'(\kp R)}c_{0,0}$, and $\beta_{n,m}=-\frac{j_n(\ks R)}{{h_n^{(1)}}(\ks R)}b_{n,m}$ for $n\ge 1$. For the other elements in $\mathcal{S}_{n,m}$, we have
$\mathcal{S}_{n}$ is given as follows
\begin{eqnarray}
	\mathcal{S}_{n}	=-\begin{bmatrix}
		d_{n}^{11} & d_{n}^{12} \\
		d_{n}^{21} & d_{n}^{22}
	\end{bmatrix}^{-1}
	\begin{bmatrix}
		e_{n}^{11} & e_{n}^{12} \\
		e_{n}^{21} & e_{n}^{22}
	\end{bmatrix},
\end{eqnarray}
where 
\begin{eqnarray}
	d_{n}^{11} & = &2\mu\kp^2 H^{(1)''}_n
	(\kp R)-\lambda\kp^2H^{(1)}_n
	(\kp R),\quad d_{n}^{12}  = -2\mu \mi n\frac{\ks R H^{(1)'}_n
		(\ks R)-H^{(1)}_n
		(\ks R)}{R^2} ,
	\\
	d_{n}^{21} & = &2\mu \mi n \frac{\kp R H^{(1)'}_n
		(\kp R)-H^{(1)}_n
		(\kp R)}{R^2},\quad d_{n}^{22}  = 2\mu \ks^2H^{(1)''}_n
	(\ks R)+\mu\ks^2H^{(1)}_n(\ks R),
\end{eqnarray}
and $e_{n}^{ij}$ is simply replacing  $H_n^{(1)}$ in $d_{n}^{ij}$ by $J_n$.
\section{Matrix Representation of $F^0$ in $\mathcal{B}_l$}\label{appB}
In this appendix, we compute the matrix representation of $F^0$ restricted in
$\mathcal{B}_l$ with the selected basis $\{h_{l,1},h_{l,2},h_{l,3},h_{l,4}\}$.
For $A=\begin{pmatrix}
	a_{11}&a_{12}\\
	a_{21}&a_{22}
\end{pmatrix}$, it holds
\begin{align*}
	&2\lambda\kp^{\frac{7}{2}}\int_{\mathbb{S}}\alpha \cdot A\alpha d s_\alpha \mathbb{P}^\top +4\mu\kp^{\frac{7}{2}}\int_{\mathbb{S}}(\alpha\cdot A\alpha)\alpha\otimes\alpha d s_\alpha \mathbb{P}^\top 
	+2\mu\ks^{\frac{7}{2}}\int_{\mathbb{S}}(\alpha^\bot\cdot A \alpha)(\alpha^\bot\otimes \alpha+\alpha\otimes \alpha^\bot) d s_\alpha \mathbb{P}^\top \\
	=&\left(2\lambda\kp^{\frac{7}{2}}\pi\begin{pmatrix}
		a_{11}+a_{22}&0\\0&a_{11}+a_{22}
	\end{pmatrix}+\mu\kp^{\frac{7}{2}}\pi\begin{pmatrix}
		3a_{11}+a_{22}&a_{12}+a_{21}\\a_{12}+a_{21}&a_{11}+3a_{22}
	\end{pmatrix}+\mu\ks^{\frac{7}{2}}\pi\begin{pmatrix}
		a_{11}-a_{22}&a_{12}+a_{21}\\a_{12}+a_{21}&a_{22}-a_{11}
	\end{pmatrix}\right)\mathbb{P}^\top\\
	=&\begin{pmatrix}
		(c_1+3c_2+c_3)a_{11}+(c_1+c_2-c_3)a_{22}&(c_2+c_3)(a_{12}+a_{21})\\
		(c_2+c_3)(a_{12}+a_{21})&(c_1+c_2-c_3)a_{11}+(c_1+3c_2+c_3)a_{22}
	\end{pmatrix}\begin{pmatrix}
		p_{11}&p_{21}\\p_{12}&p_{22}
	\end{pmatrix}.
\end{align*}
Therefore, when $A=\begin{pmatrix}
	1&0\\
	0&0
\end{pmatrix},$
\begin{align*}
	\begin{pmatrix}
		c_1+3c_2+c_3&0\\
		0&c_1+c_2-c_3
	\end{pmatrix}\begin{pmatrix}
		p_{11}&p_{21}\\p_{12}&p_{22}
	\end{pmatrix}
	=\begin{pmatrix}
		(c_1+3c_2+c_3)p_{11}&(c_1+3c_2+c_3)p_{21}\\(c_1+c_2-c_3)p_{12}&(c_1+c_2-c_3)p_{22}
	\end{pmatrix}.
\end{align*}
When $A=\begin{pmatrix}
	0&0\\
	0&1
\end{pmatrix},$
\begin{align*}
	\begin{pmatrix}
		c_1+c_2-c_3&0\\
		0&c_1+3c_2+c_3
	\end{pmatrix}\begin{pmatrix}
		p_{11}&p_{21}\\p_{12}&p_{22}
	\end{pmatrix}
	=\begin{pmatrix}
		(c_1+c_2-c_3)p_{11}&(c_1+c_2-c_3)p_{21}\\(c_1+3c_2+c_3)p_{12}&(c_1+3c_2+c_3)p_{22}
	\end{pmatrix}.
\end{align*}
When $A=\begin{pmatrix}
	0&1\\
	1&0
\end{pmatrix},$
{
	\begin{align*}
		\begin{pmatrix}
			0&2(c_2+c_3)\\
			2(c_2+c_3)&0
		\end{pmatrix}\begin{pmatrix}
			p_{11}&p_{21}\\p_{12}&p_{22}
		\end{pmatrix}
		=2(c_2+c_3)\begin{pmatrix}
			p_{12}&p_{22}\\p_{11}&p_{21}
		\end{pmatrix}.
\end{align*}}
When $A=\begin{pmatrix}
	0&1\\
	-1&0
\end{pmatrix},$
\begin{align*}
	\begin{pmatrix}
		0&0\\
		0&0
	\end{pmatrix}\begin{pmatrix}
		p_{11}&p_{21}\\p_{12}&p_{22}
	\end{pmatrix}
	=0.
\end{align*}
Thus
\begin{align}\label{matf}
	\mathbb{F}=\begin{pmatrix}
		(c_1+3c_2+c_3)p_{11}&(c_1+c_2-c_3)p_{11}&2(c_2+c_3)p_{12}&0\\
		(c_1+c_2-c_3)p_{22}&(c_1+3c_2+c_3)p_{22}&2(c_2+c_3)p_{21}&0\\
		\frac{(c_1+3c_2+c_3)p_{21}+(c_1+c_2-c_3)p_{12}}{2}&\frac{(c_1+c_2-c_3)p_{21}+(c_1+3c_2+c_3)p_{12}}{2}&(c_2+c_3)(p_{11}+p_{22})&0\\
		\frac{(c_1+3c_2+c_3)p_{21}-(c_1+c_2-c_3)p_{12}}{2}&\frac{(c_1+c_2-c_3)p_{21}-(c_1+3c_2+c_3)p_{12}}{2}&(c_2+c_3)(p_{22}-p_{11})&0
	\end{pmatrix},
\end{align}
where $c_1=2\lambda\kp^{\frac{7}{2}}\pi,c_2=\mu\kp^{\frac{7}{2}}\pi,c_3=\mu\ks^{\frac{7}{2}}\pi$.

\end{document}